\documentclass[11pt,leqno,oneside,letterpaper]{amsart}

\usepackage{amssymb,amsmath,amsfonts,latexsym, amscd, amsthm, url} 
\usepackage[letterpaper,left=1.75truein, top=1truein]{geometry}
\usepackage[usenames, dvipsnames]{color}
\usepackage[colorlinks,linkcolor=blue,citecolor=blue]{hyperref}
\usepackage[normalem]{ulem}
\usepackage{caption,subcaption,graphicx,epsfig}
\usepackage{tikz-cd,calc}
\usetikzlibrary{knots}
\usetikzlibrary{decorations.markings}
\definecolor{light-gray}{gray}{0.92}
\definecolor{ultra-light-gray}{gray}{0.97}

\newtheorem{theorem}{Theorem}[section]
\newtheorem{lemma}[theorem]{Lemma}
\newtheorem{proposition}[theorem]{Proposition}
\newtheorem{definition}[theorem]{Definition}

\newtheorem{corollary}[theorem]{Corollary} 

\theoremstyle{definition}
\newtheorem{remark}[theorem]{Remark} 
\newtheorem{remarks}[theorem]{Remarks}

\newtheorem{problem}[theorem]{Problem} 
 
\newtheorem{conjecture}[theorem]{Conjecture}

\theoremstyle{cases}

\numberwithin{subcase}{case} \numberwithin{subsubcase}{subcase}
\numberwithin{equation}{subsection}

\newcommand{\thomeo}{{\widetilde{\rm{Homeo}}}}

\def\Sm{S_n} % used to denote the m-fold cyclic cover of n-punctured disk $D_m$. 
\def\bSm{\bar{S}_n}

\newcommand\blfootnote[1]{%
  \begingroup
  \renewcommand\thefootnote{}\footnote{#1}%
  \addtocounter{footnote}{-1}%
  \endgroup
} % Footnote without footnotemark \blfootnote

\title{Taut foliations in branched cyclic covers and left-orderable groups}

\author[Steven Boyer]{Steven Boyer}
\thanks{Steven Boyer was partially supported by NSERC grant RGPIN 9446-2013}
\address{D\'epartement de Math\'ematiques, Universit\'e du Qu\'ebec \`a Montr\'eal, 201 avenue du Pr\'esident-Kennedy, Montr\'eal, QC H2X 3Y7.}
\email{boyer.steven@uqam.ca}
\urladdr{http://www.cirget.uqam.ca/boyer/boyer.html}

\author[Ying Hu]{Ying Hu}
\thanks{Ying Hu was partially supported by a CIRGET postdoctoral fellowship}
\address{D\'epartement de Math\'ematiques, Universit\'e du Qu\'ebec \`a Montr\'eal, 201 avenue du Pr\'esident-Kennedy, Montr\'eal, QC H2X 3Y7.}
\email{ying.hu@cirget.ca}
\urladdr{https://sites.google.com/site/yinghumath/}

\begin{document}

\maketitle 

\begin{center}
\today
\end{center}

\begin{abstract}
We study the left-orderability of the fundamental groups of cyclic branched covers of links which admit co-oriented taut foliations. In particular we do this for cyclic branched covers of fibred knots in integer homology $3$-spheres and cyclic branched covers of closed braids. The latter allows us to complete the proof of the L-space conjecture for closed, connected, orientable, irreducible $3$-manifolds containing a genus $1$ fibred knot. We also prove that the universal abelian cover of a manifold obtained by generic Dehn surgery on a hyperbolic fibred knot in an integer homology $3$-sphere admits a co-oriented taut foliation and has left-orderable fundamental group, even if the surgered manifold does not, and that the same holds for many branched covers of satellite knots with braided patterns. A key fact used in our proofs is that the Euler class of a universal circle representation associated to a co-oriented taut foliation coincides with the Euler class of the foliation's tangent bundle. Though known to experts, no proof of this important result has appeared in the literature. We provide such a proof in the paper.
\end{abstract}


\section{Introduction}
\label{sec:intro}

In this paper we study the left-orderability of the fundamental groups of rational homology $3$-spheres $M$ which admit co-oriented taut foliations. Our primary motivation is the {\it L-space conjecture}: 

\begin{conjecture}[Conjecture 1 in \cite{BGW}, Conjecture 5 in \cite{Ju}] 
\label{conj: lspace}
{\it Assume that $M$ is a closed, connected, irreducible, orientable $3$-manifold. Then the following statements are equivalent.  

$(1)$ $M$ is not a Heegaard Floer $L$-space,

$(2)$ $M$ admits a co-orientable taut foliation,

$(3)$ $\pi_1(M)$ is left-orderable. }

\end{conjecture}
The conjecture is known to hold in a variety of situations, most notably when $M$ has positive first Betti number (\cite{BRW,Ga1}), or is a non-hyperbolic geometric $3$-manifold (\cite{BGW, BRW, LS}), or is a graph manifold ({\cite{BC, HRRW}). Condition (2) of the conjecture is known to imply condition (1) (\cite{{OS1},KR2, {Bn}}). Gordon and Lidman introduced the term {\it excellent} for manifolds satisfying conditions (2) and (3), and therefore (1), of the conjecture, and {\it total L-space} for manifolds satisfying neither (1) nor (3), and therefore neither (2). It is clear that Conjecture \ref{conj: lspace} holds for manifolds which are either excellent or total L-spaces and that the conjecture is equivalent to the statement that a closed, connected, irreducible, orientable $3$-manifold is either excellent or a total $L$-space. 

Given a closed, connected, irreducible, orientable $3$-manifold $M$, the available techniques for verifying that $M$ satisfies conditions (1) and (2) of Conjecture \ref{conj: lspace} are far in  advance of those available for verifying (3).  An equivalent condition for (3) is the existence of a non-trivial homomorphism $\pi_1(M) \to \hbox{Homeo}_+(\mathbb R)$ (\cite[Theorem 1.1]{BRW}), but these are difficult to construct in general. One method for producing them is to consider a non-trivial representation $\rho: \pi_1(M) \to PSL(2, \mathbb R)$ whose Euler class vanishes (cf. \S \ref{sec: euler class}). Such a $\rho$ lifts to a representation $\pi_1(M) \to \widetilde{SL_2} \leq \hbox{Homeo}_+(\mathbb R)$, and so $\pi_1(M)$ is left-orderable. A drawback of this approach is that it gives no insight into potential connections between condition (3) and conditions (1) and (2). To address this point, suppose that $M$ satisfies (2) and let $\rho: \pi_1(M) \to \hbox{Homeo}_+(S^1)$ be a non-trivial representation obtained through Thurston's universal circle construction applied to a co-oriented taut foliation on $M$ (cf. \S \ref{sec:taut foliations circle}). As before, there is a characteristic class $e(\rho) \in H^2(M)$ whose vanishing implies the left-orderability of $\pi_1(M)$ (see \S \ref{sec: euler class}). It is known that $e(\rho)$ coincides with the Euler class of the foliation's tangent bundle (see Proposition \ref{prop:Euler class}), and while the latter does not always vanish, one goal of this paper is to show that it does in topologically interesting situations. In particular, we use this approach to investigate Conjecture \ref{conj: lspace} in the context of manifolds obtained as branched covers of knots and links in rational homology $3$-spheres. 

Gordon and Lidman initiated such a study for links in $S^3$ (\cite{GLid1, GLid2}), focusing on torus links and certain families of satellite knots, including cables. Here we will be mainly concerned with cyclic branched covers of hyperbolic links. In this case, the cyclic branched covers are almost always hyperbolic (\cite{BPH,Dun}). 

Hyperbolic $2$-bridge knots form one of the simplest families of hyperbolic knots and various aspects of Conjecture \ref{conj: lspace} have been studied for their branched covers. For instance, work of Dabkowski, Przytycki, and Togha \cite{DPT} combines with that of Peters \cite{Pe} to show that the branched covers of many genus one $2$-bridge knots, including the figure eight knot, are total L-spaces. The second named author showed that for large $n$, the fundamental group of the $n$-fold branched cyclic cover of the $(p,q)$ $2$-bridge knot is left-orderable if $q \equiv 3$ (mod $4$) \cite{Hu}. More generally, Gordon showed that the same conclusion holds for any $2$-bridge knot with non-zero signature \cite{Gor}.

Before we state our results, we introduce some notation and terminology. See \S \ref{sec: background} for the details.  

Given a $3$-manifold $V$ with a connected toroidal boundary, a slope on $\partial V$ is a $\partial V$-isotopy class of essential simple closed curves contained in $\partial V$. We identify slopes with $\pm$-classes of primitive elements of $H_1(\partial V)$, in the usual way, and often represent them by primitive classes $\alpha \in H_1(\partial V)$. The Dehn filling of $V$ determined by a slope $\alpha$ on $\partial V$ will be denoted by $V(\alpha)$. 

Let $K$ be an oriented null-homologous knot in an oriented rational homology sphere $M$. We use $X(K)$ to denote its exterior in $M$ and $\mu, \lambda \in H_1(\partial X(K))$ to denote, respectively, the longitudinal and meridional classes of $K$ (cf. \S \ref{subsec:knot exterior}). Since $K$ is null-homologous, $\{\mu, \lambda\}$ is a basis of $H_1(\partial X(K))$. 

For each $n \geq 1$, $X_n(K) \to X(K)$ will be the canonical $n$-fold cyclic cover of $X(K)$ and $\Sigma_n(K)\to M$ the associated $n$-fold cyclic cover branched over $K$. 
There is a basis $\{\mu_n, \lambda_n\}$ of $H_1(\partial X_n(K))$ where the image of $\mu_n$ in $H_1(\partial X(K))$ is $n \mu$ and that of $\lambda_n$ is $\lambda$. By construction, $\Sigma_n(K)=X_n(K)(\mu_n)$
(\S \ref{sec:cyclic branched cover}). 

Given a fibred knot $K$ in an irreducible rational homology sphere, we use $c(h)$ to denote the fractional Dehn twist coefficient of  its monodromy $h$ (\S \ref{sec:fractional Dehn twist coefficient}). When $K$ is hyperbolic and $c(h)\neq 0$, work of Roberts (\cite{Rob}) can be used to show that if $n|c(h)| \geq 1$, the $n$-fold cyclic cover branched cover of such $K$ admits co-oriented taut foliations (\cite[Theorem 4.1]{HKM2}).  We use the universal circle construction to show that under the same conditions, the branched covers have left-orderable fundamental groups: 

\begin{theorem} 
\label{thm:conjecture fibre knots} 
Let $K$ be a hyperbolic fibred knot in an oriented integer homology $3$-sphere $M$ with monodromy $h$.

$(1)$ $\Sigma_n(K)$ is excellent for $n|c(h)| \geq 1$. In particular, if the fractional Dehn twist coefficient $c(h) \ne 0$ and $g$ is the genus of $K$, then $\Sigma_n(K)$ is excellent for $n\geq 2(2g-1)$.

$(2)$ More generally, for $n \geq 1$, $X_n(K)(\mu_n + q \lambda_n)$ is excellent whenever $|nc(h) - q| \geq 1$.
\end{theorem} 
For a fixed $n$, there are at most two values of $q$ for which $|nc(h) - q| < 1$ and if two, they are successive integers. Such exceptional values of $q$ are necessary as, for instance, $X_n(K)(\mu_n + q \lambda_n)$ could have a finite fundamental group. Compare Corollary \ref{cor:universal abelian cover}.

It is known that the fractional Dehn twist coefficients of the monodromies of hyperbolic, fibred, strongly quasipositive knots are non-zero (\cite{Hed, HKM1}). In particular, this is true for $K$ an L-space knot, as they are fibred and strongly quasipositive (cf. \cite[Theorem 1.2]{Hed}, \cite[Corollary 1.3]{Ni} and the calculations of \cite{OS2}).

\begin{corollary} 
\label{cor:cyclic branched covers of SQP}
Suppose that $K$ is a hyperbolic, fibred, strongly quasipositive knot with monodromy $h$. Then $\Sigma_n(K)$ is excellent for $n \geq \frac{1}{|c(h)|}$. In particular, $\Sigma_n(K)$ is excellent  if $n\geq 2(2g-1)$.  
\qed 
\end{corollary}

Boileau, Boyer and Gordon have investigated the $n$-fold branched cyclic covers of strongly quasipositive knots \cite{BBG} and have shown that in the fibred case they are not L-spaces for $n \geq 6$.   Since $c(h)$ can be arbitrarily small for such knots, the disparity between the sufficient condition $n \geq 6$ for condition (1) of the conjecture to hold and $n \geq \frac{1}{|c(h)|}$ for conditions (2) and (3) to hold is arbitrarily large. A major challenge is to develop techniques to bridge this gap. \\

\begin{remarks} \label{rem: main theorem}$\;$ 

(1) Theorem \ref{thm:conjecture fibre knots} and its corollaries (Corollary \ref{cor:universal abelian cover}, Corollary \ref{cor: knots rational homology sphere}) hold for hyperbolic fibred knots in oriented rational homology spheres under the assumption that the Euler class of the tangent plane bundle of the fibring of the exterior of the knot is zero (Proposition \ref{prop: e=0 and c>1 implies lo}).

(2)  In Theorem {\ref{thm:conjecture fibre knots}}, the inequality $|nc(h) - q| \geq 1$ can be recast in terms of the {\it distance} $\Delta(\alpha, \beta)$ between slopes $\alpha, \beta$ on $\partial X(K)$. Thinking of $\alpha$ and $\beta$ as primitive classes in $H_1(\partial X(K))$ and using $\alpha \cdot \beta$ to denote their algebraic intersection number, $\Delta(\alpha, \beta)$ is defined to be $|\alpha \cdot \beta|$. If $c(h) = \frac{a}{b}$ where $a, b$ are coprime integers, then the {\it degeneracy slope} of $K$ is represented by the primitive class $\delta = b \mu + a \lambda$\ (\cite{GO, KR1}). Then $|nc(h) - q| < 1$ if and only if $\Delta(n \mu + q \lambda, \delta) = |na - qb| < |b| = \Delta(\lambda, \delta)$. Thus the theorem says that $X_n(K)(\mu_n + q \lambda_n)$ is excellent if $\Delta(n \mu + q \lambda, \delta) \geq  \Delta(\lambda, \delta)$.
\end{remarks}
\vspace{-.2cm}

The {\it universal abelian cover} of a manifold $W$ is the regular cover $\widetilde{W} \to W$ corresponding to the abelianisation homomorphism $\pi_1(W) \to H_1(W)$. It is simple to see that if $\gcd(n, q) = 1$, there is a universal abelian cover $X_n(K)(\mu_n + q \lambda_n) \to X(K)(n\mu + q \lambda)$. 

\begin{corollary} 
 \label{cor:universal abelian cover}
Let $K$ be a hyperbolic fibred knot in an integer homology $3$-sphere with monodromy $h$. Given coprime integers $n \geq 1$ and $q$, then the universal abelian cover of $X(K)(n\mu+q\lambda)$ is excellent for 
$q \not \in \left\{ 
\begin{array}{cl} 
\{nc(h)\}  & \hbox{ if } nc(h) \in \mathbb Z \\ 
\{\lfloor nc(h) \rfloor, \lfloor nc(h) \rfloor+ 1\}  & \hbox{ if } nc(h) \not  \in \mathbb Z  
\end{array} \right.$.  
\end{corollary}

Corollary \ref{cor:universal abelian cover} is striking in that it says that the universal abelian cover of the generic Dehn surgery on a hyperbolic fibred knot in an integer homology $3$-sphere is excellent even when the surgered manifold is not. Consider, for instance, a hyperbolic L-space knot $K \subset S^3$. Up to replacing $K$ by its mirror image, we can suppose that $n/q$-surgery of $K$ is an L-space if and only if $n/q \geq 2g(K) - 1$. The corollary implies that if $n/q \geq 2g(K) - 1$, then avoiding the specified values of $q$, $n/q$-surgery of $K$ is a non-excellent manifold whose universal abelian cover is excellent.

Assuming the truth of Conjecture \ref{conj: lspace}, the corollary holds for all hyperbolic knots in the $3$-sphere. For instance, if $K$ is a non-fibred hyperbolic knot in $S^3$, it admits no non-trivial surgeries which yield L-spaces (\cite{Ghi, Ni}). Conjecture \ref{conj: lspace} then implies that for $n$ and $q$ as in the corollary, the rational homology sphere $X(K)(n \mu + q \lambda)$ admits a co-orientable taut foliation. Hence the same is true for its universal abelian cover $X_n(K)(\mu_n + q \lambda_n)$. This cover also has a left-orderable fundamental group, and is therefore excellent, by Remark \ref{rem: rhouniv non-trivial} and \cite[Lemma 3.1]{BRW}.

\begin{conjecture}
{\it Let $n, q$ be coprime integers with $nq \ne 0$ and let $K$ be a hyperbolic knot in $S^3$. If the universal abelian cover of $X(K)(n \mu + q \lambda)$ is not excellent, then $K$ is fibred and if $h$ is its monodromy,  $q \in \left\{ 
\begin{array}{cl} 
\{nc(h)\}  & \hbox{ if } nc(h) \in \mathbb Z \\ 
\{\lfloor nc(h) \rfloor , \lfloor nc(h) \rfloor + 1\}  & \hbox{ if } nc(h) \not  \in \mathbb Z  
\end{array} \right.$.}
\end{conjecture}

\begin{problem}
{\it Determine necessary and sufficient conditions for the universal abelian cover of an irreducible rational homology $3$-sphere $M$ to be excellent. }
\end{problem}

For instance, is the existence of a representation $\pi_1(M) \to \mbox{{\rm Homeo}}_+(S^1)$ with non-abelian image necessary and sufficient for the universal abelian cover of an irreducible rational homology $3$-sphere $M$ to be excellent?

Fix a knot $K$ in an integer homology $3$-sphere $M$ and coprime integers $p > 0$ and $q$. Let $m$ be a positive integer and set $n = mp$. We can generalize Theorem \ref{thm:conjecture fibre knots}(1) and Corollary \ref{cor:universal abelian cover} by considering the orbifold with underlying space $X(K)(p\mu + q \lambda)$ and singular set the core of the filling solid torus with isotropy groups $\mathbb Z/m$. Here $H_1(\mathcal{O}) \cong \mathbb Z/n$ and the universal abelian cover of $\mathcal{O}$ corresponds to an $n$-fold cyclic cover $X_n(K)(\mu_n + mq \lambda_n) \to X(K)(p\mu + q \lambda)$ branched over the core of the $(p\mu + q \lambda)$-filling torus with branching index $m$. When $p = 1$ and $q = 0$ this is the branched cover $\Sigma_n(K) \to M$.

\begin{corollary} 
\label{cor: knots rational homology sphere} 
Let $K$ be a hyperbolic fibred knot in an integer homology $3$-sphere $M$ with monodromy $h$ and consider coprime integers $p > 0$ and $q$ as well as a positive integer $m\geq 1$. The universal abelian cover of the orbifold with underlying space $X(K)(p\mu + q \lambda)$ and singular set the core of the filling solid torus with isotropy group $\mathbb Z/m$ is an excellent $3$-manifold if $m|pc(h) - q| \geq 1$. 
\end{corollary}

We also have results on cyclic branched covers of non-fibred hyperbolic links in $S^3$. Here is a special case of Theorem \ref{thm:taut foliation in cyclic covers of a braid}.

\begin{theorem} 
\label{thm:conjecture cyclic braids}  
Let $b \in B_{m}$ be an odd-strand pseudo-Anosov braid and let $c(b)$ denote its fractional Dehn twist coefficient. If $|c(b)| \geq 2$, then all even order cyclic branched covers of $\hat{b}$ are excellent.
\end{theorem}

\begin{remark}
It is useful to note that under the hypothesis that $|c(b)| \geq 2$, work of Ito and Kawamuro (\cite[Theorem 8.4]{IK}) implies that $b$ is a pseudo-Anosov braid if and only if $\hat b$ is a hyperbolic link.
\end{remark}

Theorem \ref{thm:conjecture cyclic braids} combines with results of Baldwin (\cite{Bal}) and Li-Watson (\cite{LW}) to prove that: 

\begin{theorem} 
\label{thm:lspace genus one open book decomposition}
Conjecture \ref{conj: lspace} holds for irreducible $3$-manifolds which admit genus one open book decompositions with connected binding. 
\end{theorem} 

In its turn, Theorem \ref{thm:lspace genus one open book decomposition} combines with Theorem \ref{thm:conjecture fibre knots} to determine precisely which branched covers of genus one fibred knots $K$ are excellent and which are total L-spaces. To describe this, let $T_1$ be the fibre of such a knot. It is known that the mapping class group ${\rm Mod}(T_1)$ is generated by two right-handed Dehn twists $T_{c_1}$ and $T_{c_2}$ (cf. \S \ref{sec:L-space conjecture genus one open books}, especially Figure \ref{fig:double cover 3-punctured disk}). Let 
 $$\delta = (T_{c_1}T_{c_2})^3$$ 
and note that $\delta^2$ is the right-handed Dehn twist along $\partial T_1$.}
 
\begin{corollary} 
\label{cor: branched cover genus 1} 
Suppose that $K$ is a genus one fibred knot with monodromy $h$ in a closed, connected, orientable and irreducible $3$-manifold $M$. Then  for each $n \geq 2$, $\Sigma_n(K)$ is either excellent or a total L-space. Further, $\Sigma_n(K)$ is a total L-space if and only if 

$(1)$ $h$ is pseudo-Anosov, $c(h) = 0$, and $n \geq 2$.

$(2)$ $h$ is periodic and up to replacing it by a conjugate homeomorphism, $h$ and $n$ are given by 
$$h = \left\{ \begin{array}{rll} 
T_{c_1}^{-1} T_{c_2}^{-1} & \hbox{ and } & n \leq 5 \\ 
\delta T_{c_1}^{-1} T_{c_2}^{-1} &  \hbox{ and } & n= 2 \\ 
T_{c_1}^{-2} T_{c_2}^{-1}  & \hbox{ and } & n \leq 3 \\
\delta T_{c_1}^{-2} T_{c_2}^{-1} &  \hbox{ and } & n \leq 3 \\
T_{c_1}^{-3} T_{c_2}^{-1}  & \hbox{ and } & n = 2 \\
\delta T_{c_1}^{-3} T_{c_2}^{-1}  & \hbox{ and } &  n \leq 5 .
\end{array} \right.$$
\end{corollary}

Next we consider satellite links.

In \cite{GLid1}, Gordon and Lidman studied the cyclic branched covers of $(p,q)$-cable knots in $S^3$. These are satellite knots whose patterns are $(p,q)$-torus knots embedded standardly as a $q$-braid in a solid torus. They showed that the $n$-fold cyclic branched covers of $(p,q)$-cable knots are always excellent, except possibly for the case $n=q=2$ (\cite[Theorem 1.3]{GLid1}). In the latter case they showed that the $2$-fold branched covers of a $(p,2)$-cable knots are never L-spaces \cite[Theorem 1]{GLid2}, and hence the truth of Conjecture \ref{conj: lspace} would imply that they are excellent.  

\begin{conjecture} {\rm (Gordon-Lidman)}
 \label{conj:cyclic satellite}
{\it The $n$-fold cyclic branched cover  of a prime, satellite knot is excellent.} 
\end{conjecture}

Satellite links whose patterns are closed braids and whose companions are fibred are a particularly interesting class to investigate as, for instance, all satellite L-space knots in $S^3$ fall into this category (\cite[Theorem 7.3, Theorem 7.4]{BM}; also see \cite[Theorem 35]{HRW} and \cite[Proposition 3.3]{Hom}). Theorem \ref{thm:satellite c(h)>0 n>>0} and Corollary \ref{cor:satellite c(h)>0 n>>0} verify special cases of Conjecture \ref{conj:cyclic satellite}. 

\begin{theorem} 
\label{thm:satellite c(h)>0 n>>0}
Assume that $L$ is a satellite link in an integer homology $3$-sphere $M$ whose pattern is contained in its solid torus as the closure of an $m$-strand pseudo-Anosov braid and whose companion is a fibred hyperbolic knot in $M$ with monodromy $h$. 

$(1)$ If $c(h)=0$, then the $n$-fold cyclic branched cover of $L$ is excellent whenever $\gcd(m,n)=1$.

$(2)$ If $c(h)\neq 0$, then the $n$-fold cyclic branched cover of $L$ is excellent when $\gcd(m,n) =1$ and $n\geq \frac{2}{|c(h)|}$. 

\end{theorem}

By Proposition \ref{prop:lower bound FDTC}, if $c(h) \ne 0$, then $|c(h)| \geq \frac{1}{2(2g(C)-1)}$ where $g(C)$ is the genus of the companion knot $C$ in Theorem \ref{thm:satellite c(h)>0 n>>0}. Hence the condition $n\geq \frac{2}{|c(h)|}$ in Theorem \ref{thm:satellite c(h)>0 n>>0}(2) holds if $n\geq 4(2g(C)-1)$.

\begin{corollary} 
\label{cor:satellite c(h)>0 n>>0} 
Assume that $L$ is a satellite link in an integer homology $3$-sphere $M$ whose pattern is contained in its solid torus as the closure of an $m$-strand pseudo-Anosov braid and whose companion is a fibred hyperbolic knot. Then the $n$-fold cyclic branched cover of $L$ is excellent when $\gcd(m,n) =1$ and $n \gg 0$. 
\qed
\end{corollary}

Consider an L-space satellite knot $K$. Baker and Motegi have shown that the pattern is a closed braid \cite[\S 7]{BM}. Further, Hanselman, Rasmussen and Watson \cite{HRW} have shown that the companion is also an L-space knot. Hence the companion knot $C$ is fibred and strongly quasipositive (cf. \cite[Theorem 1.2]{Hed}, \cite[Corollary 1.3]{Ni} and the calculations of \cite{OS2}), so its fractional Dehn twist coefficient is non-zero (cf. \cite{Hed, HKM2}). Up to replacing $K$ by its mirror image, we can suppose that the fractional Dehn twist of the monodromy  of the companion knot $C$ is positive. 

Boileau, Boyer and Gordon have shown that the cyclic branched covers of satellite L-space knots are never L-spaces \cite[Corollary 6.4]{BBG}. 
In the case that both pattern and companion are hyperbolic, and the fractional Dehn twist coefficient of the pattern braid is nonnegative (cf. \cite[Question 1.8]{Hom}), Theorem \ref{thm: satellite knot c(b) and c(h) nonnegative} shows that $\Sigma_n(K)$ is excellent whenever $n$ is relatively prime to the braid index of the pattern.

\begin{theorem} 
\label{thm: satellite knot c(b) and c(h) nonnegative}
Assume that $L$ is a satellite link in an integer homology $3$-sphere $M$ whose pattern is contained in its solid torus as the closure of an $m$-strand pseudo-Anosov braid $b$ and whose companion is a fibred hyperbolic knot with monodromy $h$. Suppose that the fractional Dehn twist coefficients $c(b)$ and $c(h)$ are non-negative. Then for $n\geq 2$ relatively prime to $m$, the $n$-fold cyclic branched cover of $L$ is excellent.
\end{theorem}

%\begin{theorem} 
%\label{thm: satellite knot c(b) and c(h) nonnegative}
%Assume that $L$ is a satellite link in an integer homology $3$-sphere $M$ whose pattern $P$ is hyperbolic and contained in its solid torus as the closure of an $m$-strand braid $b$, and whose companion is a fibred hyperbolic knot $C$ in $M$ with fibre $S$ and monodromy $h$. Suppose that the fractional Dehn twist coefficients $c(b)$ and $c(h)$ are non-negative. Then for $n\geq 2$ relatively prime to $m$, the $n$-fold cyclic branched cover of $L$ is excellent.
%\end{theorem}

Here is the plan of the paper. In \S \ref{sec: background} we introduce background material and notational conventions. Section \ref{sec:braids MCG} covers some basic concepts regarding mapping class groups and braids. Section \ref{sec:fractional Dehn twist coefficient} introduces fractional Dehn twist coefficients from two perspectives: isotopies (\S \ref{subsec:FDTC isotopy}) and translation numbers (\S \ref{subsec:FDTC  translation number}). The Euler classes of representations with values in $\hbox{Homeo}_+(S^1)$ and of oriented circle bundles are defined and related in \S  \ref{sec: euler class}. Section \ref{sec:taut foliations circle} is devoted to a description of the universal circle and the universal circle representation associated to a rational homology $3$-sphere $M$ endowed with a co-oriented taut foliation. In \S \ref{sec:euler class universal circle} we give a detailed proof of the fact, due to Thurston, that the Euler class of the universal circle representation coincides with that of the associated foliation's tangent bundle (Proposition \ref{prop:Euler class}), and \S \ref{sec: lo and taut foliations} uses this to deduce the left-orderabilty of $\pi_1(M)$ when this Euler class vanishes. The material of the previous sections is combined  in \S \ref{sec: fdtc and lo} to study the left-orderability of $3$-manifolds given by open books. In particular, Theorem \ref{thm:conjecture fibre knots}, Corollary \ref{cor:universal abelian cover} and Corollary \ref{cor: knots rational homology sphere} are proved here. In \S \ref{sec:cyclic closed braids} we prove Theorems \ref{thm:conjecture cyclic braids} and  \ref{thm:taut foliation in cyclic covers of a braid}, which are used in \S \ref{sec:L-space conjecture genus one open books} to deduce Theorem \ref{thm:lspace genus one open book decomposition} and Corollary \ref{cor: branched cover genus 1}. Finally in \S \ref{sec:LO cyclic cover satellite knots}, we apply the results of \S \ref{sec: fdtc and lo} and \S \ref{sec:cyclic closed braids} to study cyclic branched covers of satellite knots in order to prove Theorems \ref{thm:satellite c(h)>0 n>>0} and \ref{thm: satellite knot c(b) and c(h) nonnegative}.

{\bf Acknowledgement}.
The authors would like to thank Jonathan Bowden for pointing out the possibility of avoiding the use of \cite[Theorem 1.4]{HKP} in our arguments and for discussions which led us to add Lemma \ref{lem:Euler class vanishes links} in its place. They also thank Bill Menasco for an enlightening correspondence concerning relations between fractional Dehn twist coefficients and open book foliations. Finally they thank the anonymous referee for suggestions which led to an improved exposition. 

\section{Some background results, terminology and notation} 
\label{sec: background}
 
We set some conventions in this section which will be used throughout the paper.  

\subsection{Link exteriors in rational homology spheres}
\label{subsec:knot exterior} 

Let $M$ be an oriented rational homology $3$-sphere and $L$ be an oriented null-homologous link in $M$. We use $N(L)$ to denote a closed tubular neighbourhood of $L$ and $X(L) = \overline{M \setminus N(L)}$ to denote the exterior of $L$ in $M$. 

If $L = \bigsqcup_i K_i$ is the decomposition of $L$ into its component knots, then $N(L) = \bigsqcup_i N(K_i)$ where $N(K_i)$ is a tubular neighbourhood of $K_i$. 

A {\it meridional disk} of $K_i$ is any essential properly embedded disk in $N(K_i)$ which is oriented so that its intersection with the oriented knot $K_i$ is positive.

The {\it meridional slope} of $K_i$ is represented by a primitive class $\mu_i \in H_1(\partial N(K_i))$ corresponding to the oriented boundary of a meridional disk of $K_i$. 

A {\it meridional class of $K_i$} in $H_1(X(L))$ is the image of $\mu_i$ under the inclusion-induced homomorphism $H_1(\partial N(K_i)) \to H_1(X(L))$. 

The assumption that $L$ is null-homologous implies that there is a compact, connected, oriented surface $S$ properly embedded in $X(L)$ which intersects each component $\partial N(K_i)$ of $\partial X(L)$ in an oriented simple closed curve $\lambda_i$ isotopic in $N(K_i)$ to $K_i$. It is clear from the construction that 
$$\mu_i \cdot \lambda_i = 1$$
for each $i$.

In the case that $L$ is a knot, $\lambda_1$ represents the longitudinal class of $L = K_1$ in $H_1(\partial X(K_1))$.

\begin{lemma} 
\label{lemma: homology of exterior}
Suppose that $K$ is a null-homologous knot in a rational homology $3$-sphere $M$ with exterior $X(K)$. Then $H_1(X(K)) \cong H_1(M) \oplus \mathbb Z$ where the second factor is generated by a meridional class of $K$. Further, the inclusion-induced homomorphism $H^2(M) \to H^2(X(K))$ is an isomorphism. 
\end{lemma}

\begin{proof}
Excision implies that 
$$H_r(M, X(K)) \cong H_r(N(K), \partial N(K)) \cong \left\{ \begin{array}{ll} \mathbb Z & \hbox{ if } r = 2,3 \\ 0 & \hbox{ otherwise} \end{array} \right.$$ 
where $H_2(M, X(K)) \cong \mathbb Z$ is generated by the class $\eta$ carried by a meridional disk of $N(K)$.  Then the exact sequence of the pair $(M, X(K))$ yields a short exact sequence 
\begin{equation} 
\label{eqn: sequence}
0 \to H_2(M, X(K)) \xrightarrow{\partial} H_1(X(K)) \to H_1(M) \to 0 
\end{equation} 
where $\partial(\eta)$ is a meridional class $\mu$ of $K$. Since $K$ is null-homologous in $M$, there is a properly embedded, compact, connected, oriented surface $S$ in $X(K)$ whose boundary represents the longitudinal class $\lambda$ of $K$ in $H_1(\partial X(K))$. Then $\partial([S]) \cdot \mu = \lambda \cdot \mu = \pm 1$, where $[S] \in H_2(X(K), \partial X(K))$ corresponds to the fundamental class of $S$. Hence the homomorphism $H_1(X(K)) \to H_2(M, X(K)), \alpha \mapsto (\alpha \cdot [S]) \eta$, splits the sequence (\ref{eqn: sequence}) up to sign, which proves the first assertion of the lemma. 

For the second, consider the connecting map $H^1(X(K)) \xrightarrow{\delta} H^2(M, X(K))$ from the cohomology exact sequence of the pair $(M, X(K))$. Excision shows that 
$$H^2(M, X(K)) \cong H^2(N(K), \partial N(K)) \cong \hbox{Hom}(H_2(N(K), \partial N(K)), \mathbb Z) \cong \mathbb Z$$ 
is generated by the homomorphism which takes the value $1$ on the class $\eta \in H_2(M, X(K)) = H_2(N(K), \partial N(K))$. On the other hand, since $H_1(X(K)) \cong H_1(M) \oplus \mathbb Z$ where the $\mathbb Z$ factor is generated by $\partial \eta$, if $\nu \in \hbox{Hom}(H_1(X(K)), \mathbb Z) \cong H^1(X(K))$ is the homorphism which takes the value $1$ on a meridian of $K$ and $0$ on $H_1(M)$, then $\delta(\nu)$ is a generator of $H^2(N(K), \partial N(K))$. Thus $\delta$ is surjective. It then follows from the exact cohomology sequence of the pair $(M, X(K))$ that the homomorphism $H^2(M) \to H^2(X(K))$ is an isomorphism, which completes the proof. 
\end{proof}

\subsection{Cyclic branched covers of null-homologous links}
\label{sec:cyclic branched cover}
Given a null-homologous oriented link $L= \bigsqcup_i K_i$ in an oriented rational homology sphere $M$ and compact, connected, oriented surface $S$ properly embedded in $X(L)$ as above, let $[S] \in H_2(X(L), \partial X(L))$ correspond to the fundamental class of $S$. For each $n \geq 1$, the epimorphism 
$$H_1(X(L))  \xrightarrow{\;\; \alpha \mapsto \alpha \cdot [S] \;\;}  \mathbb{Z} \xrightarrow{\tiny \;\; \hbox{(mod $n$) reduction} \;\;}  \mathbb Z/n$$ 
determines an $n$-fold cyclic cover 
$$X_n(L) \to X(L)$$
and an $n$-fold cyclic cover 
$$(\Sigma_n(L), \widetilde L) \xrightarrow{\;\; p \;\;} (M, L)$$ 
branched over $L$. 

The link $\widetilde L$ decomposes into components $\widetilde L = \bigsqcup_i \widetilde K_i$ 
where $\widetilde K_i=p^{-1}(K_i)$. Similarly, its closed tubular neighbourhood $N(\widetilde{L}) = \overline{\Sigma_n(L) \setminus X_n(L)}$ splits into components $N(\widetilde{L}) = \bigsqcup_i N(\widetilde K_i)$ where $N(\widetilde K_i)$ is a tubular neighbourhood of $\widetilde K_i$. 

For each $i$ there is a basis $\{\widetilde \mu_{i}, \widetilde \lambda_{i}\}$ of $H_1(\partial N(\widetilde K_i))$ determined by the property that 
$$\widetilde \mu_i \xrightarrow{\;\;\; p_* \;\;} n \mu_i $$ 
and
$$\widetilde \lambda_i \xrightarrow{\;\;\; p_* \;\;}  \lambda_i$$ 
The surface $S$ lifts to a properly embedded surface $\widetilde S \subset X_n(L)$ which intersects $\partial \widetilde N(K_i)$  in an oriented simple closed curve representing $\widetilde \lambda_i$.

By construction, $\Sigma_n(L)$ is the $(\widetilde \mu_1, \widetilde  \mu_2, \ldots, \widetilde  \mu_n)$-Dehn filling of $X_n(L)$.

\subsection{Lifting contact structures to branched covers}
\label{subsec: contact}

 Let $M$ be an oriented rational homology $3$-sphere and $L$ an oriented null-homologous link in $M$. Fix a compact, connected, oriented surface $S$ properly embedded in $X(L)$ and let $p: (\Sigma_n(L), X_n(L), \widetilde L) \xrightarrow{\;\; p \;\;} (M, X(L), L)$ be as above, where $n \geq 1$. 

Let $\xi=\ker(\alpha)$ be positive contact structure on $M$ determined by a smooth, nowhere zero $1$-form $\alpha$ and suppose that $L$ is a positively transverse to $\xi$. There is a lift of $\xi$ to $\Sigma_n(L)$, denoted by $\widetilde{\xi}$, which is the kernel of the pull-back form $p^*(\alpha)$ on $X_n(L)$ and is positively transverse to $\widetilde{L}$ (cf. \cite[\S 2.5]{HKP}, \cite[Theorem 7.5.4]{Gei}). More precisely, $\xi$ can be constructed as follows. 

Recall that $L=\bigsqcup_i K_i$ and $\widetilde L=\bigsqcup_i \widetilde K_i$ where $\widetilde K_i=p^{-1}(K_i)$. For each $i$, there is a suitable tubular neighborhoods $N(K_i)$ and $N(\widetilde{K}_i)$, and cylindrical coordinates $(r,\theta,z)$ and $(\tilde{r},\tilde{\theta},\tilde{z})$ over the tubular neighborhoods $N(K_i)$ and $N(\widetilde{K}_i)$ respectively such that the contact form $\alpha$ restricted to $N(K_i)$ is in the standard form $\alpha|_{N(K_i)}=dz+r^2d\theta$ (\cite[Example 2.5.16]{Gei})  and the cyclic branched cover $p$ restricts to $N(\widetilde{K}_i)\setminus \widetilde{K}_i$ sends $(\tilde{r},\tilde{\theta},\tilde{z})$ to $(r,n\theta, z)$. The pull-back $p^*(\alpha|_{{N(K_i)}\setminus K_i})=d\tilde{z}+n\tilde{r}^2d\tilde{\theta}$ is a contact form over $N(\widetilde{K}_i) \setminus \widetilde{K}_i$ which extends smoothly to $N(\widetilde{K}_i)$ by letting $\tilde\alpha|_{\tilde K_i} = d\tilde z$. Extending $p^*(\alpha|_{\Sigma_n(L)\setminus \widetilde{L}})$ in this way to the entire tubular neighborhood, we produce the desired contact form on $\Sigma_n(L)$, denoted by $\widetilde \alpha$. Let  $\widetilde{\xi}=\ker(\widetilde \alpha)$.

\subsection{Fibred knots and open books} 
In this section, we review the definitions of fibred knots and open books. See \cite[\S 10H and \S 10K]{Rol} for the details. 

An oriented knot $K$ in $M$ is called {\it fibred} with {\it fibre} $S$ if $S$ is a compact, connected, orientable surface properly embedded in $X(K)$ which has connected boundary and there is a locally-trivial fibre bundle $X(K) \to S^1$ with fibre $S$. Note that $S \cap \partial X(K)$ carries the longitudinal slope $\lambda$ of $K$. 
A {\it monodromy} of $K$ is an orientation-preserving homeomorphism $h:S \to S$ such that $h|_{\partial S}$ is the identity,  $X(K) \cong (S \times I)/((x,1) \sim (h(x), 0))$, and if $x \in \partial S$, then the loop on $\partial X(K)$ determined by $\{x\} \times I$ carries the meridional slope $\mu$ of $K$. If $K$ is a knot in a rational homology $3$-sphere, its monodromy is well-defined up to conjugation and an isotopy fixed on $\partial S$. See \cite[Proposition 5.10]{BZ} for a proof of this claim in the case that $K \subset S^3$. 

Conversely, given an orientation-preserving homeomorphism $h$ of a compact, connected, orientable surface $S$ with connected boundary which restricts to the identity on $\partial S$, there is a well-defined closed, connected, orientable $3$-manifold $M$ obtained from the Dehn filling of $(S \times I)/((x,1) \sim (h(x), 0))$ along the slope determined by the image of  $\{x\} \times I$ for $x \in \partial S$. The core of the filling solid torus is a knot $K$ in $M$ which is fibred with fibre $S$ and monodromy $h$. The meridian of $K$ is carried by the image of $\{x\} \times I$. The pair $(S, h)$ is called an {\it open book} decomposition of $M$ with {\it binding} $K$. 

\section{Mapping class groups and closed braids}
\label{sec:braids MCG}
Throughout this section $S$ will denote an $m$-punctured ($m \geq 0$) smooth orientable compact surface with nonempty boundary. All diffeomorphisms of $S$ will be assumed to be orientation-preserving. 

We use ${\rm Mod}(S)$ to denote the mapping class group of isotopy classes of diffeomorphisms of $S$ which restrict to the identity on $\partial S$. Isotopies are assumed to be fixed on $\partial S$. 
 
From time to time we will identify an element of ${\rm Mod}(S)$ with one of its representative diffeomorphisms, though only when discussing properties held by all such representatives. 

\subsection{The Nielsen-Thurston classification of mapping classes}  
\label{subsec:nielsen-thurson classification}
A homeomorphism $\varphi:S\rightarrow S$ is called pseudo-Anosov if it preserves a pair $(\mathcal{F}^s, \mu^s)$ and $(\mathcal{F}^u,\mu^u)$ of mutually transverse, measured, singular foliations on $S$, and there is a number $\lambda>1$ such that  $\varphi$ scales the transverse measure $\mu^s$ by $\lambda^{-1}$ and the transverse measure $\mu^u$ by $\lambda$. Here $(\mathcal{F}^s, \mu^s)$  and $(\mathcal{F}^u,\mu^u)$  are called 
 the {\it stable foliation} and {\it unstable foliation} of $\varphi$ respectively. We refer the reader to \cite[Chapter 13]{FM} for more precise details on pseudo-Anosov homeomorphisms as well as other results stated in this section.

By the Nielsen-Thurston classification \cite{Thu2}, each element $f$ in ${\rm Mod}(S)$ is freely isotopic to a map $\varphi:S\rightarrow S$ which is either 
\begin{itemize}
\vspace{-.2cm} 

\item a periodic diffeomorhpism, i.e. $\varphi^n=1$ for some $n>0$, or

\vspace{.2cm} \item a reducible diffeomorphism, i.e. there exists a nonempty collection $\mathcal C=\{c_1,\cdots c_r\}$ of pairwise disjoint essential  simple closed curves in $S$ such that $f(\mathcal{C})= \mathcal{C}$, or  

\vspace{.2cm} \item a pseudo-Anosov homeomorphism.

\end{itemize}
The homeomorphism $\varphi$ is called a {\it Nielsen-Thurston representative} of $f$; $f$ is called periodic, reducible or pseudo-Anosov if its Nielsen-Thurston representative has the corresponding property.

It is known that a pseudo-Anosov mapping class is neither periodic nor reducible \cite{Thu2}. A fundamental result of Thurston is that the interior of the mapping torus of $f \in \rm{Mod}(S)$ is a finite volume hyperbolic manifold if and only if $f$ is pseudo-Anosov \cite{Thu1}. It contains an essential torus if $f$ is reducible and it is a Seifert fibred manifold if $f$ is periodic. 

\subsection{The braid group $B_m$ and ${\rm Mod}(D_m)$}
\label{subsec:braid group mcg} 
We use $B_m$ to denote the group of isotopy classes of smooth $m$-strand braids, where each strand of a braid is oriented upward (Figure \ref{fig:fig34}(A)). Let $\sigma_i$ be the standard $i^{th}$ Artin generator of the braid group $B_m$, $i=1,\cdots, n-1$ (Figure \ref{fig:fig34}(A)). 

\medskip 

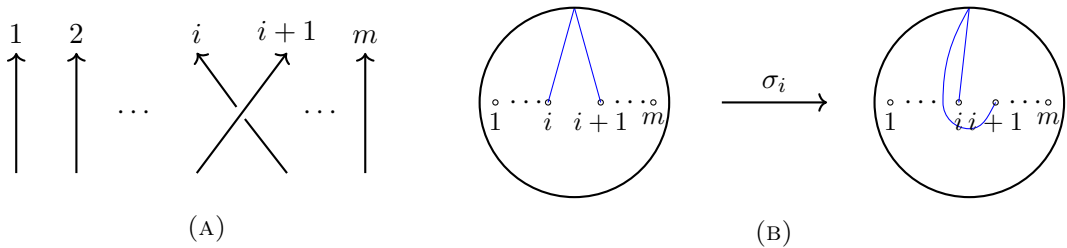
\begin{figure}[ht]
\centering
 \begin{subfigure}{.4\textwidth}
  \begin{tikzpicture}[scale=0.8]
  \draw [thick, ->] (0.5,1) -- (0.5,3);
  \node [above] at (0.5,3) {$1$}; 
  \draw [thick, ->] (1.5,1) -- (1.5,3);
  \node [above] at (1.5,3) {$2$}; 
  \node [right] at (2,2) {$\cdots$}; 
  \draw [thick, ->] (3.5,1) -- (5,3);
  \draw [thick] (5,1) -- (4.3,1.9);
  \draw [thick,->] (4.16,2.1) -- (3.5,3);
  \node [right] at (5.1,2) {$\cdots$};
  \draw [thick,->] (6.3, 1) -- (6.3,3);
  \node [above] at (6.3,3) {$m$};
  \node [above] at (5,3) {$i+1$}; 
  \node [above] at (3.5,3) {$i$};
  \node at (0,1) {\large \color{white}{d}};
  \end{tikzpicture}
  \caption{}
\end{subfigure}%
\begin{subfigure}{.55\textwidth}
  \centering
  \begin{tikzpicture}[scale=0.7]
  %\draw [help lines] (0,0) grid (12,5);
  \node [below] at (1,3) {\small $1$};
  \draw (1,3) circle (0.05);
  \node [right] at (1.07,3) {$\cdots$};
  \draw (2,3) circle (0.05);
  \node [below] at (2,3) {\small $i$};
  \draw (3,3) circle (0.05);
  \node [below] at (3,3) {\small $i+1$};
  \node [right] at (3.05,3) {$\cdots$};
  \draw (4,3) circle (0.05);
  \node [below] at (4,3) {\small $m$};
  \draw [blue] (2.5, 4.8) -- (2,3);
  \draw [blue] (2.5,4.8) -- (3,3);
  \draw [thick] (2.5,3) circle (1.8);
 % The middle one
  \draw [thick, ->] (5.3,3) -- (7.3,3);  
  \node [above] at (6.3,3) {$\sigma_i$}; 
 % the right-hand side of the circle.
 \node [below] at (8.5,3) {\small $1$};
  \draw (8.5,3) circle (0.05);
  \node [right] at (8.57,3) {$\cdots$};
  \draw (9.8,3) circle (0.05);
  \draw [blue] (10,4.8)--(9.8,3);
  \draw [blue] (10,4.8) to  [in=90,out=240] (9.5,3) to [in=180, out=270] (10,2.5) to [out=0,in=240] (10.5,3);
  \draw (10.5,3) circle (0.05);
  \node [right] at (10.55,3) {$\cdots$};
   \draw (11.5,3) circle (0.05);
   \node [below] at (11.5,3) {\small $m$};
   \node [below] at (9.8,3) {\small $i$};
  \node [below] at (10.5,3) {\small $i+1$};
  \draw [thick] (10,3) circle (1.8);
  %\node at (0,0.5) {\large \color{white}{d}};
 \end{tikzpicture}
  \caption{}
\end{subfigure}
\caption{(A) $\sigma_i$ in $B_m$; (B) $\sigma_i$ in ${\rm Mod}(D_m)$}
\label{fig:fig34}
\end{figure}

The braid group $B_m$ is isomorphic to the mapping class group ${\rm Mod}(D_m)$, where $D_m$ denotes the $m$-punctured disk obtained by removing $m$ points on the real line from the interior of the unit disk $D^2$. See \cite[Chapter 9]{FM} for instance.  We identify these two groups through the following correspondence.  Given an $m$-strand braid the corresponding diffeomorphism of  $D_m$ is obtained by sliding the $m$-punctured disk along the braid from the bottom to the top (see Figure \ref{fig:fig34}).  
The product of two braids $b_1$ and $b_2$ is the braid obtained by placing $b_1$ on the top of $b_2$.  When $b_1$ and $b_2$ are viewed as diffeomorphisms of the punctured disk $D_m$, we have $b_1b_2(x)=b_1(b_2(x))$ for all $x\in D_m$.

A braid $b\in {\rm Mod}(D_m)$ is called {\it pseudo-Anosov}, respectively {\it periodic}, respectively {\it reducible}, if it is freely isotopic to a homeomorphism of $D_m$ with the corresponding property.   

\subsection{Hyperbolic links as  closures of pseudo-Anosov braids} 
\label{subsec:hyperbolic pseudo-Anosov}
The {\it closure} of a braid $b$, denoted $\widehat{b}$, is an oriented link in $S^3$ obtained by closing the braid $b$ as illustrated in Figure \ref{fig:branched_cover_braid_fig0}.  A classical theorem of Alexander \cite{Al} asserts that for any oriented link $L$ in $S^3$, there is an $m \geq 1$ and a braid $b \in B_m$ such that $L$ is isotopic to $\widehat b$. 

\begin{figure}[ht]
\centering
\begin{tikzpicture}[scale=0.85]
 %\draw [help lines] (0,0) grid (14,6);
 \draw [gray, thick] (4, 4.5) ellipse (0.75 and 0.3);
 \draw [gray,thick] (4, 1) ellipse (0.75 and 0.3);
  \draw [gray, thick,->] (4,0.7) -- (4.05,0.7);
 \draw [gray, thick] (3.25,1) -- (3.25, 4.5);
 \draw [gray,thick] (4.75,1) -- (4.75, 4.5);
  \draw [gray,thick, ->] (4.75,2.75) -- (4.75,2.8);
 \draw [thick] (3.5,2) rectangle (4.5,3.5);
 \node at (4,2.75) {$b$}; 
 \begin{scope}[thick,decoration={
    markings,
    mark=at position 0.5 with {\arrow{>}}}
    ] 
    \draw[postaction={decorate}] (3.6,3.5)--(3.6,4.5);
    \draw[postaction={decorate}] (3.85,3.5)--(3.85,4.5);
    \draw[postaction={decorate}] (4.4,3.5)--(4.4,4.5);
     \draw[postaction={decorate}] (3.6,1)--(3.6,2);
    \draw[postaction={decorate}] (3.85,1)--(3.85,2);
    \draw[postaction={decorate}] (4.4,1)--(4.4,2);
\end{scope}
\filldraw (4,3.8) circle (0.015);
\filldraw (4.1,3.8) circle (0.015);
\filldraw (4.2,3.8) circle (0.015);
% dots on the bottom
\filldraw (4,1.65) circle (0.015);
\filldraw (4.1,1.65) circle (0.015);
\filldraw (4.2,1.65) circle (0.015);
% draw points on the top and bottom of the cylinder
\filldraw (3.6,1) circle (0.035);
\filldraw (3.85,1) circle (0.035);
\filldraw (4.4,1) circle (0.035);
\filldraw (3.6,4.5) circle (0.035);
\filldraw (3.85,4.5) circle (0.035);
\filldraw (4.4,4.5) circle (0.035);
%% mid arrow 
\draw [->] (5.5, 2.8) -- (7.5,2.8);
% closed braid starts
\draw [gray, dashed] (8.25, 2.65) to [out=90, in=180] (9,3) to [out=0, in=90] (9.65, 2.65);
\draw [thick] (8.5,2) rectangle (9.5,3.5);
  \begin{scope}[thick,decoration={
    markings,
    mark=at position 0.1 with {\arrow{>}}}
    ] 
    \draw[postaction={decorate}] (9.4,3.5) to [out=90, in=270] (9.4,3.7) to [out=90, in=180]  (9.8,4) to [out=0, in=90] (10.2,2.75) to [out=270, in=0] (9.8, 1.5) to [out=180, in=270] (9.4,1.8) to [out=90,in=270] (9.4,2);
    \draw[postaction={decorate}] (8.85,3.5) to [out=90, in=270] (8.85, 3.7) to [out=90, in=180] (9.8, 4.5) to [out=0,in=90] (10.55, 2.75) to [out=270, in=0] (9.8, 1.1) to [out=180, in=270] (8.85,1.8) to [out=90,in=270](8.85,2);
    \draw[postaction={decorate}] (8.6,3.5) to [out=90,in=270] (8.6,3.65) to [out=90,in=180] (9.8, 4.75) to [out=0,in=90] (10.75,2.75) to [out=270,in=0] (9.8,0.9) to [out=180, in=270] (8.6, 1.85) to [out=90,in=270] (8.6, 2);
\end{scope}
 \begin{scope}[decoration={
    markings,
    mark=at position 0 with {\arrow{>}}}
    ] 
 \draw [blue, postaction={decorate}] (9.65,2.9) -- (9.65, 2.93);
\end{scope}
\draw [gray,thick] (8.25, 2.8) to [out=90, in=270] (8.25, 3.5) to [out=90, in=180] (9.7, 5) to [out=0, in=90] (11.1,2.75) to [out=-90, in=0] (9.7, 0.5) to [out=180, in=270] (8.25, 1.7) to [out=90, in=-90] (8.25,2.8);
\node at (9,2.75) {$b$}; 
% meridian curve
\begin{scope}[decoration={
    markings,
    mark=at position 0.5 with {\arrow{>}}}
    ] 
\draw [gray,thick, postaction={decorate}] (8.25, 2.65) to [out=270, in=180] (9,2.3) to [out=0, in=270] (9.65, 2.65);
\end{scope}
% inner curve v
\draw [blue, thick] (9.8, 2.75) ellipse (0.15 and 0.7);
\node at (9.8, 3.65) {{\color{blue}\small $\nu$}};
% dots on the top
\filldraw (9,3.8) circle (0.015);
\filldraw (9.1,3.8) circle (0.015);
\filldraw (9.2,3.8) circle (0.015);
% dots on the bottom
\filldraw (9,1.65) circle (0.015);
\filldraw (9.1,1.65) circle (0.015);
\filldraw (9.2,1.65) circle (0.015);
\end{tikzpicture}

\caption{}
\label{fig:branched_cover_braid_fig0}
\end{figure}

Recall that a link $L$ in $S^3$ is {\it hyperbolic} if its exterior $S^3\setminus L$ is hyperbolic, i.e., it admits a complete finite volume Riemannian metric of constant curvature $-1$. Ito has shown \cite[Theorem 1.3]{Ito2} that if the absolute value of the Dehornoy floor (Definition \ref{def:Dehornoy floor}) of $b \in B_m$ is at least $2$ and $\widehat b$ is a knot, then $\widehat b$ is hyperbolic if and only if $b$ is pseudo-Anosov. More recently, Ito and Kawamuro have shown \cite[Theorem 8.4]{IK} that if the absolute value of the fractional Dehn twist coefficient of $b$ is larger than $1$, then $\hat b$ is a hyperbolic link if and only if $b$ is a pseudo-Anosov braid.

\section{Fractional Dehn twist coefficients}
\label{sec:fractional Dehn twist coefficient}

In this section we suppose that $S$ is an oriented hyperbolic surface with nonempty geodesic boundary. Given $h: S \rightarrow S$ representing an element of ${\rm Mod}(S)$, let $H_t: S\rightarrow S$ denote a free isotopy between $H_0 = h$ and its  Nielsen-Thurston representative $H_1 = \varphi$. 

We are interested in the fractional Dehn twist coefficient of $h$ with respect to a boundary component of $S$. Intuitively, this is a rational number representing the amount of twisting $\partial S$ undergoes during the isotopy $H_t$ from $h$ to $\varphi$. The concept was introduced in \cite{HKM1} to study the tightness of the contact structure supported by an open book $(S, h)$. When $h$ is pseudo-Anosov, it is closely related to the degeneracy slope of a pseudo-Anosov homeomorphism \cite{GO}. If $\partial S$ is connected, it can be used to formulate a convenient criterion, due to Honda, Kazez and Matic, for the existence of co-oriented taut foliations on the open book $(S,h)$  (cf. Theorem \ref{thm:cgeq1}). 

We give two equivalent definitions of the fractional Dehn twist coefficient in Section \ref{subsec:FDTC isotopy} and Section \ref{subsec:FDTC translation number} below, and will take advantage of both points of view. For simplicity, we also assume that $\partial S$ is connected and leave the simple task of extending the definition to the case that $\partial S$ is not connected to the reader (also see \cite{HKM1}).

The following theorem summarizes results from \cite{HKM2}. See Theorem 4.1, Theorem 4.3, and Lemma 4.4 of that paper for the details. (We remark that the results of \cite{Bn, KR2} are needed for the proof of \cite[Theorem 4.3]{HKM2}.) 

\begin{theorem}[\cite{HKM2}] \label{thm:cgeq1}
Assume that $(S,h)$ is an open book decomposition of a closed oriented $3$-manifold $M$, where $\partial S$ is connected and $h$ is freely isotopic to a pseudo-Anosov homeomorphism. If the fractional Dehn twist coefficient of $h$ satisfies $c(h)\geq 1$, then there exists a co-orientable taut foliation on $M$ which is transverse to the binding of $(S,h)$ and is homotopic to the contact structure supported by $(S,h)$. 
\end{theorem}

\subsection{Fractional Dehn twist coefficients via isotopies}
\label{subsec:FDTC isotopy}
Here, we define the fractional Dehn twist coefficient following the ideas of \cite{HKM1}, where the pseudo-Anosov case is dealt with. The extension to all Nielson-Thurston types is immediate. 

Let $C \cong S^1$ denote the boundary of $S$ and fix a periodic orbit $\{p_0,\cdots, p_{n-1}\} \subset C$ of the Nielsen-Thurston representative $\varphi$ of $h$ as follows: When $\varphi$ is pseudo-Anosov, we may choose $\{p_0,\cdots, p_{n-1}\}$ to be a subset of the singular points on $C$ of the stable singular foliation of $\varphi$ (see \cite[\S 5.1]{FLP} for examples of singularities on the boundary).  In the case that $\varphi$ is reducible, let  $S_0$ be the subsurface of $S$ that contains $C$. We require that $\varphi|_{S_0}$ is either pseudo-Anosov or periodic.

Assume that $p_0,\cdots, p_{n-1}$ are indexed cyclically according to the orientation on $C$ induced by that on $S$. Since the set $\{p_0,\cdots, p_{n-1}\}$ is preserved under $\varphi$, there exists an integer $k\in \{0, 1, \cdots, n-1\}$ such that $\varphi(p_0)=p_k$. Then $H_t|_{p_0}: [0,1]\rightarrow C$ defines a path on the boundary component $C$ connecting $H_0(p_0)=p_0$ to $H_1(p_0)=p_k$.

The orientation of $C$ determines an oriented subarc $\gamma_{p_0p_k}$ of $C$ from $p_0$ to $p_k$. Let $\bar \gamma_{p_0p_k}$ denote the same arc with the opposite orientation. Then $H_t(p_0) * \bar \gamma_{p_0p_k}$, the concatenation of the paths $H_t(p_0)$ and $\bar \gamma_{p_0p_k}$, is a loop in $C$ based at $p_0$. Hence there is a unique integer $m$ such that
\begin{equation}
 [H_t|_{p_0}\cdot\bar \gamma_{p_0p_k}]= [C]^m \in \pi_1(C, p_0)
 \label{equ:defhomotopyI}
\end{equation}
where $[C]$ is the generator of $\pi_1(C, p_0) \cong \mathbb Z$ determined by the orientation of $C$.

\begin{definition}[\cite{HKM1}]
The {\it fractional Dehn twist coefficient} $c(h)$ of the diffeomorphism $h$ is defined to be 
$$c(h)=m+\frac{k}{n}.$$ 
  \label{def:fractional dehn twist}
\end{definition}

\begin{remark}
Since the connected components of ${\rm Homeo}(S)$ are contractible when $S$ is hyperbolic (see Theorem 1.14 in \cite{FM} and the references therein), any two paths $H_t$ and $H'_t$ between $h$ and $\varphi$ (as above) are homotopic rel $\{h, \varphi\}$. Thus the paths $H_t|_{p_0}$ and $H'_t|_{p_0}$ are homotopic rel $\{0, 1\}$, which shows that $c(h)$ is independent of the choice of $H_t$. A similar argument shows that $c(h)$ depends only on the class of $h$ in ${\rm Mod}(S)$, and so determines an invariant for each mapping class in ${\rm Mod}(S)$. 
\end{remark}

\begin{proposition}{\rm (cf. \cite[Theorem 4.4]{KR1})} \label{prop:lower bound FDTC}
Let $S$ be a compact orientable hyperbolic surface with connected boundary $C$ and let $h$ be a diffeomorphism of $S$ which restricts to the identity on $\partial S$. If $h$ is pseudo-Anosov and $c(h)\neq 0$, then    
 \begin{displaymath}
  |c(h)|\geq \frac{1}{-2\chi(S)}.
 \end{displaymath}
 \end{proposition}
 
\begin{proof}
This is a straightforward consequence of the Euler-Poincar\'e formula \cite[Proposition 11.4]{FM}. 

Let $\mathcal{F}^s$ be the stable singular foliation of the pseudo-Anosov homeomorphism $\varphi$ that is freely isotopic to $h$. By definition, $c(h)$ can be written as a possibly unreduced fraction $p/q$ with $q > 0$ being the number of the singular points of $\mathcal{F}^s$ on $C$. 
 
Let $\{x_i\}$ be the singular points of $\mathcal{F}^s$ contained in the interior of $S$. For each $i$, $n_i \geq 3$ will denote the number of prongs of $\mathcal{F}^s$ at $x_i$. Then by the Euler-Poincar\'e formula,  we have 
  \begin{displaymath}
   2(\chi(S) +1)= (2-q) + \sum_{i} (2-n_i). 
  \end{displaymath}
Since $n_i\geq 3$, $\sum_{x_i} (2-n_i) \leq 0$ with equality only if $\{x_i\}$ is empty. It follows that 
 $-2\chi(S)\geq q$. By assumption, $c(h)\neq 0$. Therefore,  
 \begin{displaymath}
   |c(h)|=\frac{|p|}{q}\geq \frac{1}{q}\geq \frac{1}{-2\chi(S)}.
 \end{displaymath}
\end{proof}

\subsection{Fractional Dehn twist coefficients via translation numbers}
\label{subsec:FDTC translation number}
Recall that the group ${\rm Homeo}_+(S^1)$ has the following central extension: 
\begin{equation}
 1 \longrightarrow \mathbb Z \longrightarrow \thomeo_+(S^1) \stackrel{\pi}{\longrightarrow} {\rm Homeo}_+(S^1) \longrightarrow 1,
\label{equ:homeoext}
\end{equation}
where $\thomeo_+(S^1)$ is the universal covering group of ${\rm Homeo}_+(S^1)$, consisting of the elements of ${\rm Homeo}_+(\mathbb{R})$ which commute with translation by $1$, which we denote by ${\rm sh}(1)$. The kernel  of the covering homomorphism $\pi$ is the group of integral translations of the real line.  

Poincar\'e showed that for $\tilde f \in \thomeo_+(S^1)$ and $x_0 \in \mathbb R$, the limit 
\begin{displaymath}
\lim_{n\to \infty} \frac{\tilde f^n(x_0)-x_0}{n} 
\end{displaymath} 
exists and is independent of $x_0$ (see \cite[\S 5]{Ghy}). He defined the {\it translation number} $\tau(\tilde f)$ of $\tilde f$ to be this common limit.

Let $\tilde S$ be the universal cover of $S$ and $\tilde{C}\subset\tilde S$ be a lift of $\partial S=C$. By construction, we can take $\tilde S$ to be a closed subset of $\mathbb H^2$ with geodesic boundary. In particular, $\tilde C$ is geodesic.  We use $\partial_{\infty} \tilde{S}$ to denote the intersection of the Euclidean closure of $\tilde S$ in $\overline{\mathbb H}^2$ with $\partial \overline{\mathbb{H}}^2$. Then $\partial \tilde S \cup \partial_{\infty} \tilde{S}$ is homeomorphic to a circle. The complement of the closure of $\tilde C$ in this circle is homeomorphic to $\mathbb{R}$ which we parameterise and orient so that the lift of the Dehn twist along $C$, denoted by $T_{\partial S}$, is the translation ${\rm sh}(1)$.  

Given any element $f\in {\rm Mod}(S)$, let $\tilde f: \tilde{S}\rightarrow \tilde{S}$ denote the unique lift of $f$ satisfying $\tilde f|_{\tilde{C}}=id_{\tilde{C}}$. This correspondence defines an embedding of groups 
\begin{displaymath}
 \iota: {\rm Mod}(S)\rightarrow \thomeo_+(S^1),
\end{displaymath}
with $\iota(T_{\partial S})={\rm sh}(1)$. It was shown in \cite{Mal} (see also \cite[Theorem 4.16]{IK}) that  the fractional Dehn twist coefficient of $h$ in ${\rm Mod}(S)$ satisfies 
\begin{displaymath}
 c(h)=\tau(\iota(h)).
\end{displaymath}

% To simplify notation we will write $\tau(\iota(h))$ as $\tau(h)$.

Here are a few properties of fractional Dehn twist coefficients inherited from those of translation numbers (cf. \cite[\S 5]{Ghy}). 

\begin{lemma} \label{lem:poincare translation number} 
The fractional Dehn twist coefficient map $c: {\rm Mod}(S) \to \mathbb Q$ takes the value $1$ on $T_{\partial S}$ and is invariant under conjugation. If $h_1, h_2 \in \hbox{Mod}(S)$ commute, then $c(h_1h_2) = c(h_1) + c(h_2)$. In particular, $c(h^n)=nc(h)$ for any  $h \in {\rm Mod}(S)$ and $n \in \mathbb{Z}$.  
\end{lemma}

\section{Euler classes of circle bundles and representations}
\label{sec: euler class}
In this section, we first review the definition of the the Euler class of an oriented $S^1$-bundle over a CW complex $X$ and how it relates to the problem of lifting a representation $\rho: \pi_1(X)\rightarrow {\rm Homeo}_+(S^1)$ to a representation into $\widetilde{\rm Homeo}_+(S^1)$ (see (\ref{equ:homeoext})).

\subsection{Euler classes of circle bundles}
\label{subsec: Euler class circle bundle}
Let $\xi$ be an oriented circle bundle $E \to X$ where $X$ is a CW complex. The Euler class $e(\xi) \in H^2(X)$ is the obstruction to finding a section of $\xi$ and its vanishing is equivalent to the triviality of $\xi$ as a bundle. A representative cocycle for $e(\xi)$ is constructed as follows. See Chapter 4 of \cite{CC}, and in particular \S 4.3 and \S 4.4, for the details.  

Since $S^1$ is a $K(\mathbb{Z},1)$, the only obstruction to the existence of a section of $\xi$ arises when one tries to extend a section over the $1$-skeleton $X^{(1)}$ of $X$ to the $2$-skeleton $X^{(2)}$. Fix a section $\sigma: X^{(1)}\rightarrow E $ and define a cellular $2$-cochain $c_\sigma: C_2(X)\rightarrow \mathbb{Z}$ as follows. Let  $\varphi_\alpha: D^2\rightarrow X^{(2)}$ be the characteristic map of a $2$-cell $e_\alpha$ and let $\xi_{D^2}=(E_{D^2}\rightarrow D^2)$ denote the pull-back of $\xi$ through $\varphi_\alpha$.  Then $\sigma$ defines a section of $\xi_{D^2}$ over $\partial D^2$. Since $D^2$ is contractible, $\xi_{D^2}$ is trivial. By fixing a trivialization $E_{D^2}\rightarrow D^2\times S^1$, one has the following composite map from $S^1$ to $S^1$
\begin{displaymath}
S^1 =  \partial D^2\rightarrow E_{D^2}\rightarrow D^2\times S^1 \rightarrow S^1.
\end{displaymath}
The value of $c_\sigma$ on $e_\alpha$ is defined to be the degree of this map. This $2$-cochain is actually a cocycle whose cohomology class $[c_\sigma]$ is independent of the choices made in its construction. Further, the class is equal to the Euler class $e(\xi)$. 

Given an oriented $2$-disk-bundle or an oriented $\mathbb R^2$-bundle, there is an associated oriented circle bundle $\xi$ over $X$. The Euler class of the $2$-disk-bundle or the $\mathbb R^2$-bundle is defined to be $e(\xi)$.  

\subsection{Euler classes and Thom classes} 
\label{subsec: euler and thom}

For later use, we record how to express the Euler class of an oriented $S^1$-bundle $\xi$ in terms of the Thom class of the associated disk bundle. For details, see \cite[\S 5.7]{Spa}, where the Thom class is referred to as the {\it orientation class} and the Euler class is referred to as the {\it characteristic class}.

Consider the mapping cylinder $D_\xi \to X$ of an oriented circle bundle $E \to X$. This is an oriented $2$-disk bundle and as such has a Thom class $u_\xi \in H^2(D_\xi, E)$ uniquely characterised by the condition that for each disk fibre $D$ of $D_\xi$, the image of $u_\xi$ under the restriction homomorphism $H^2(D_\xi, E) \to H^2(D, \partial D)$ is the orientation generator. The Euler class $e(\xi)$ of $\xi$ is the image of $u_\xi$ under the composition $H^2(D_\xi, E) \to H^2(D_\xi) \xrightarrow{\cong} H^2(X)$.

\subsection{Lifting representations with values in ${\rm Homeo}_+(S^1)$}
Fix a representation $\rho:\pi_1(X)\rightarrow {\rm Homeo}_+(S^1)$. There is an associated oriented circle bundle $E_\rho \to X$ whose total space is defined by 
 \begin{displaymath}
  E_\rho = \widetilde X \times S^1 /(x, v)\sim  (g\cdot x, \rho(g) v),
 \end{displaymath}
 where $\widetilde X$ is the universal cover of $X$. The projection map $\widetilde X \times S^1 \to \widetilde X$ descends to the bundle map $E_\rho \to X$.

\begin{lemma}[\cite{Mil} Lemma 2]
A representation $\rho: \pi_1(X) \rightarrow {\rm Homeo}_+(S^1)$ lifts to a representation $\tilde{\rho}: \pi_1(X)\rightarrow \thomeo_+(S^1)$ if and only if the Euler class of the circle bundle $E_\rho$ vanishes. 
 \qed
 \label{lem:milnor Euler class}
\end{lemma}

\begin{remark}
\label{rem: Euler class of a representation}
Each central extension of a group $G$ by an abelian group $A$ determines a class $e \in H^2(G; A)$, called the characteristic class of the extension,  and the correspondence is bijective (cf.~\cite[\S 6.1]{Ghy}). Such an extension is isomorphic to the trivial extension $1 \to A \to G \times A \to G \to 1$ if and only if its characteristic class is zero. It is known that the characteristic class $e_{S^1}$ of the extension (\ref{equ:homeoext}) generates $H^2({\rm Homeo}_+(S^1))\cong \mathbb{Z}$. See \cite[Example 2.12]{MM}. Hence given a representation $\rho: G\rightarrow {\rm Homeo}_+(S^1)$, $\rho$ admits a lift to a representation with values in $\thomeo_+(S^1)$ if and only if $e(\rho)=\rho^*(e_{S^1})= 0$. When $G = \pi_1(M)$, the two obstruction classes described above coincide. More precisely, when $M$ is irreducible, we have $e(\rho)= \pm e(E_\rho) \in H^2(M)=H^2(\pi_1(M))$.  See the proof of \cite[Lemma 2]{Mil}.
\end{remark}

\subsection{ The vanishing of the Euler class of certain lifted contact structures}
\label{subsec: euler and contact}

Let $M$ be an oriented rational homology $3$-sphere and $L= \bigsqcup_i K_i$ an oriented null-homologous link in $M$. Fix a compact, connected, oriented surface $S$ properly embedded in $X(L)$ and let 
$$(\Sigma_n(L), X_n(L), \widetilde L) \xrightarrow{\;\; p \;\;} (M, X(L), L)$$ 
be as in \S \ref{sec:cyclic branched cover}, where $n \geq 1$. 

Let $\xi=\ker(\alpha)$ be a positive contact structure on $M$ determined by a smooth, nowhere zero $1$-form $\alpha$ and suppose that $L$ is a positively transverse to $\xi$. Recall the lift $\widetilde \xi$ of $\xi$ to $\Sigma_n(L)$ described in \S \ref{subsec: contact}. 

\begin{lemma} 
\label{lem:Euler class vanishes links}
If $e(\xi) = 0$, then $e(\widetilde \xi)=0$.
\end{lemma}

\begin{proof}
To simplify notation, we write $N_i$ for $N(K_i)$ and $\widetilde N_i$ for $N(\widetilde{K}_i)$. Let $D_i$ be a meridian disk of $N_i$ oriented coherently with $K_i$ and $M$. Then $\widetilde D_i=p^{-1}(D_i)$ is an oriented meridional disk in $\widetilde N_i$. We use $\widetilde m_i$ to denote the oriented boundary of $\widetilde D_i$. By construction, the $n$-fold cyclic branched cover $\Sigma_n(L)$ is obtained from $X_n(L)$ by attaching the meridian disk $\widetilde D_i$ to $X_n(L)$ along $\widetilde m_i$ for each $i$ and then plugging the boundary of the resultant $3$-manifold with $3$-cells. 

Set $\alpha_0=\alpha$ where $\alpha$ is the smooth, nowhere zero $1$-form with $\xi = \mbox{ker}(\alpha)$ described above, and fix a one-parameter family of co-oriented $2$-plane fields $\xi_t=\ker(\alpha_t)$ ($t\in [0,1]$) on $\Sigma_n(L)$ such that $\alpha_t|_{N_i}=dz_i+(1-t)r_i^2d\theta_i$. Note that $\xi_1|_{D_i}$ is the tangent bundle  $D_i$. 

Let $\widetilde \alpha_t=p^*(\alpha_t)$ over $\Sigma_n(L)\setminus \widetilde L$ where  $\Sigma_n(L) \xrightarrow{p} M$ is the branched cover. According to our choice of tubular neighborhoods $N_i$, $\widetilde N_i$ and the coordinate systems on them (\S \ref{subsec: contact}), we have 
\begin{displaymath}
 \widetilde \alpha_t|_{\widetilde N_i\setminus \widetilde K_i}=d\tilde z_i+n(1-t)\tilde r_i^2d\tilde \theta_i.
\end{displaymath}
Then $\widetilde \alpha_t$ extends smoothly to $\Sigma_n(L)$ with $\widetilde \alpha_t=d\tilde z_i$ over $\widetilde K_i$.  Hence it defines a homotopy between co-oriented $2$-plane fields $\tilde \xi=\ker(\widetilde \alpha_0)$ and $\widetilde \xi_1=\ker(\widetilde \alpha_1)$. To show $e(\widetilde \xi)=0$, we will show $e(\widetilde \xi_1)=0$.

By assumption $e(\xi_1) = e(\xi)=0$ in $H^2(M)$ so in particular, $\xi_1$ admits a nowhere zero section $\sigma$. Since $\xi_1|_{D_i}$ is the tangent bundle of $D_i$, we may suppose that $\sigma|_{D_i}=\partial_{x_i}$, where $x_i=r_i\cos(\theta_i)$ over $D_i$). See Figure \ref{fig:branched cover of a disk}. 

Let $\widetilde \sigma: X_n\rightarrow \widetilde \xi_1$ be a section of the restriction of $\widetilde \xi_1$ to $X_n$ obtained by lifting $\sigma$. Then there is a $2$-cocycle $c_{\widetilde \sigma}$ which vanishes on $X_n$ determined by $\widetilde \sigma$ with $[c_{\widetilde\sigma}]=e(\widetilde \xi_1)$ (cf. \S \ref{subsec: Euler class circle bundle}).   

Let $i: X_n(L) \rightarrow \Sigma_n(L)$ and $j: (\Sigma_n(L), \emptyset) \to (\Sigma_n(L), X_n(L))$ be the inclusions and consider the exact sequence
\begin{equation} 
\label{equ:cohomology sequence of the pair (un)branched covers}
\ldots \rightarrow H^1(\Sigma_n(L)) \rightarrow H^1(X_n(L)) \xrightarrow{\delta} H^2(\Sigma_n(L),X_n(L)) \xrightarrow{j^*} H^2(\Sigma_n(L)) \xrightarrow{i^*} H^2(X_n(L))\rightarrow \ldots  \nonumber 
\end{equation}

Since $\widetilde \sigma$ is defined over $X_n(L)$, we have $i^*([c_{\widetilde \sigma}]) = 0$ and hence $c_{\widetilde \sigma}$ represents a cohomology class in $H^2(\Sigma_n(L),X_n(L))$, which we also denoted by $[c_{\widetilde \sigma}]$. We will show that this latter class lies in the image of $\delta$ and therefore $e(\tilde\xi_1)$, which equals $j^*([c_{\widetilde \sigma}])$, is $0$.

Note that as $H_1(\widetilde N,\partial \widetilde N)\cong H^2(\widetilde N)=0$ we have 
$$H^2(\Sigma_n(L),X_n(L)) \cong H^2(\widetilde N, \partial \widetilde N) \cong {\rm Hom}(H_2(\widetilde N, \partial \widetilde N),\mathbb{Z}) \cong \oplus_i  {\rm Hom}(H_2(\widetilde N_i, \partial \widetilde N_i),\mathbb{Z}),$$
where $\widetilde N=\bigsqcup_i \widetilde N_i$, as above.
It follows that $[c_{\widetilde \sigma}] \in H^2(\Sigma_n(L),X_n(L))$ is determined by the value of $c_{\widetilde \sigma}$ on the classes $[\widetilde D_i]$ carried by the fundamental classes of the disks $\widetilde D_i$.

\vspace{1cm}

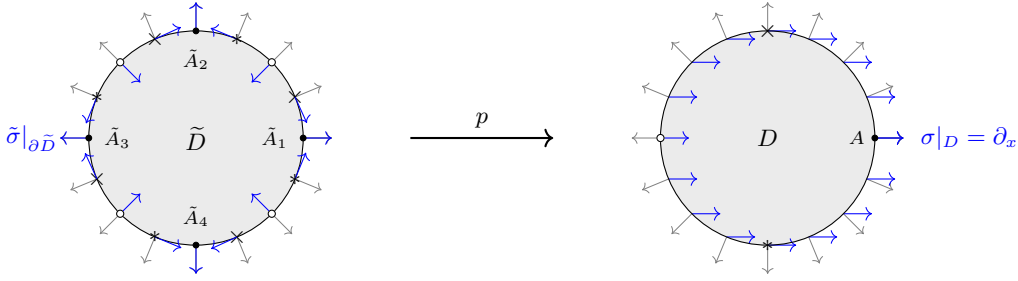
\begin{figure}[ht]

 \begin{tikzpicture}[scale=0.95]
   \centering
   %\draw [help lines] (0,0) grid (15, 6);
   %%%% left disk %%% 
    \filldraw[light-gray] (3,3) circle(1.5); 
    \foreach \x in {0,22.5,...,360} {    
                \draw [help lines,->] ({3+1.5*cos(\x)}, {3+1.5*sin(\x)}) -- ({3+1.9*cos(\x)}, {3+1.9*sin(\x)});                          
        }
     \draw (3,3) circle(1.5); 
    \foreach \x in {0, 90, ..., 270}{
    	\draw [blue,->] ({3+1.5*cos(\x)}, {3+1.5*sin(\x)}) -- ({3+1.9*cos(\x)}, {3+1.9*sin(\x)});    
	\filldraw ({3+1.5*cos(\x)}, {3+1.5*sin(\x)}) circle (0.04);
	\node at ({3+1.5*cos(\x+22.5)}, {3+1.5*sin(\x+22.5)}) {\footnotesize $\times$};
	\draw [blue,->] ({3+1.5*cos(\x+22.5)},{3+1.5*sin(\x+22.5)}) -- ({3+0.4*cos(\x-65.5)+1.5*cos(\x+22.5)}, {3+0.4*sin(\x-65.5)+1.5*sin(\x+22.5)});
	\draw [blue, <-] ({3+1.1*cos(\x+45)}, {3+1.1*sin(\x+45)}) -- ({3+1.5*cos(\x+45)}, {3+1.5*sin(\x+45)});
	\filldraw [white] ({3+1.5*cos(\x+45)}, {3+1.5*sin(\x+45)}) circle (0.05);
	\draw  ({3+1.5*cos(\x+45)}, {3+1.5*sin(\x+45)}) circle (0.05);
	\draw [blue,->] ({3+1.5*cos(\x+67.5)},{3+1.5*sin(\x+67.5)}) -- ({3+0.4*cos(\x+157.5)+1.5*cos(\x+67.5)}, {3+0.4*sin(\x+157.5)+1.5*sin(\x+67.5)});
	\node at ({3+1.5*cos(\x+67.5)}, {3+1.5*sin(\x+67.5)}) {\footnotesize $\ast$};
	}
	%% labels for the left disk
     \node at ({3+1.1*cos(0)}, {3+1.1*sin(0)}) {\tiny $\tilde A_1$};
     \node at ({3+1.1*cos(90)}, {3+1.1*sin(90)}) {\tiny $\tilde A_2$};
     \node at ({3+1.1*cos(180)}, {3+1.1*sin(180)}) {\tiny $\tilde A_3$};
     \node at ({3+1.1*cos(270)}, {3+1.1*sin(270)}) {\tiny $\tilde A_4$};
     \node [blue] at (0.7,3) {\footnotesize $\tilde\sigma|_{\partial\widetilde D}$};
       \node at (3,3) {\footnotesize $\widetilde D$};
   %%%% right disk %%% 
    \filldraw[light-gray] (11,3) circle(1.5); 
     \foreach \x in {0,22.5,...,360} {    
                \draw [help lines,->] ({11+1.5*cos(\x)}, {3+1.5*sin(\x)}) -- ({11+1.9*cos(\x)}, {3+1.9*sin(\x)});   
                \draw [->, blue] ({11+1.5*cos(\x)}, {3+1.5*sin(\x)}) -- ({11.4+1.5*cos(\x)}, {3+1.5*sin(\x)});                          
        }
          \draw (11,3) circle(1.5); 
   % A 
    \filldraw ({11+1.5*cos(0)}, {3+1.5*sin(0)}) circle (0.04);
    \node [left] at ({11+1.5*cos(0)}, {3+1.5*sin(0)}) {\tiny $A$};
   % B
    \node at ({11+1.5*cos(90)}, {3+1.5*sin(90)}) {{\footnotesize $\times$}};
   % C
    \filldraw [white]({11+1.5*cos(180)}, {3+1.5*sin(180)}) circle (0.05);
    \draw ({11+1.5*cos(180)}, {3+1.5*sin(180)}) circle (0.05);
   % D
     \node at ({11+1.5*cos(270)}, {3+1.5*sin(270)}) {{\footnotesize $\ast$}};
     \node [blue,right] at (13,3) {\footnotesize $\sigma|_D=\partial_x$};  
   
  %% Label for the right disk
  \node at (11,3) {\footnotesize $D$};

      %%% middle arrow %%% 
   \draw [thick,->] (6,3) --(8,3);
   \node [above] at (7,3) {\footnotesize $p$};
 \end{tikzpicture}
 \caption{\footnotesize The normal vector field $\partial_r$ along $\partial D$ lifts to the normal vector field $\partial_{\tilde r}$ along $\partial \tilde D$. Also for any given point $\tilde x\in \partial \tilde D$, the angle between ${\partial_{\tilde r}}|_{\tilde x}$ and $\tilde \sigma|_{\tilde x}$ is the same with the angle between $\partial_r|_{p(\tilde x)}$ and $\sigma|_{p(\tilde x)}$. Based on these two observations, it is easy to draw $\tilde \sigma$ along $\partial \tilde D$.  In this figure, we illustrate $\tilde \sigma$ alone the boundary of a meridional disk $\tilde D$ for $4$-fold cyclic branched covers. It is also easy to see that from the point $\tilde A_1$ to $\tilde A_2$, the vector field $\tilde \sigma$ rotates by an angle of $(2\pi-\frac{2\pi}{4})$ clockwise. Hence the total rotation of $\tilde \sigma$ along $\partial \tilde D$ is $-\frac{2\times 3\pi}{4}\times 4=-3\times 2\pi$. Therefore, by the construction of $c_{\tilde \sigma}$, we have $c_{\tilde \sigma}([\tilde D])=-3$.}
  \label{fig:branched cover of a disk}
\end{figure}

Figure \ref{fig:branched cover of a disk} illustrates the calculation of the value $c_{\widetilde \sigma}$ on a meridional disk. In particular, by our choice of $\widetilde \sigma$, it follows that the value $c_{\widetilde\sigma}([\widetilde D_i])=1-n$ is independent on $i$. On the other hand, if we denote by $u \in H^1(X_n(L))$ the Poincar\'e dual of the element of $H_2(X_n(L), \partial X_n(L))$ carried by the fundamental class of the lift $\widetilde S$ of a Seifert surface $S$, our assumptions on $S$ and $L$ imply that $\delta(u) \in H^2(\Sigma_n(L), X_n(L))$ evaluates to $1$ on each $[\widetilde D_i]$.  It follows that  $[c_{\widetilde \sigma}] \in H^2(\Sigma_n(L), X_n(L))$ is $(1-n)\delta(u)$. In particular, it lies in the image of $\delta$, which completes the proof.  
\end{proof}

\section{Universal circle actions}
\label{sec:taut foliations circle}

Throughout this section and the next, $M$ will denote a closed connected oriented $3$-manifold and $\mathcal{F}$ a topological co-oriented taut foliation on $M$. Such foliations are known to be isotopic to foliations whose leaves are smoothly immersed and whose tangent planes vary continuously across $M$ (\cite{Cal1}). We assume below that $\mathcal{F}$ satisfies this degree of  smoothness. Consequently, the tangent planes of $\mathcal{F}$ determine a $2$-plane subbundle $T\mathcal{F}$ of $TM$. Orient $T\mathcal{F}$ so that its orientation together with the co-orientation of $\mathcal{F}$ determines the orientation of $M$. 

Suppose that there is a Riemannian metric $g$ on $M$ whose restriction to the leaves of $\mathcal F$ has constant curvature $-1$. By this we mean that the restriction of $g$ to each plaque of $\mathcal F$ satisfies this curvature condition. Ostensibly, this appears to be a strong constraint on $M$ and $\mathcal{F}$. It precludes, for instance, the existence of a leaf of $\mathcal{F}$ which is homeomorphic to either a $2$-sphere or torus. However, work of Plante and Candel (\cite{Can,Pla}) shows that when $M$ is a rational homology $3$-sphere, such a metric always exists. See the proof of Theorem \ref{prop:taut foliation left orderability}.

Given such a metric $g$ on $M$, Thurston constructed a circle $S^1_{univ}$ associated to $\mathcal F$ and a non-trivial homomorphism $\rho_{univ}: \pi_1(M) \rightarrow {\rm Homeo}_+(S^1_{univ})$ determined by the geometry of $g$. In this section, we review the construction of $\rho_{univ}$ following the approach found in \cite{CD} (see also \cite{Cal2}).  

\subsection{Bundles from circles at infinity}
\label{subsubsec:circle bundles at infinity}
Let $(M, g, \mathcal F)$ be as above. As $\mathcal F$ is taut, the inclusion map induces an injection from the fundamental group of each of its leaves to $\pi_1(M)$ (\cite{Nov}). Hence each leaf of the pull-back foliation $\widetilde{\mathcal{F}}$ on the universal cover $\widetilde M$ of $M$ is simply-connected. 

The {\it leaf space} $\mathcal{L}$ of $\widetilde{\mathcal{F}}$ is the quotient space of $\widetilde{M}$ obtained by collapsing each leaf of $\widetilde{\mathcal{F}}$ to a point. The simple-connectivity of $\widetilde M$ implies that transversals to $\widetilde{\mathcal F}$ map homeomorphically to their images in $\mathcal L$ and as such, the co-orientation on $\mathcal{F}$ determines an orientation on these images. Globally, $\mathcal{L}$ is an oriented, though not necessarily Hausdorff, $1$-manifold (cf. \cite[Corollary D.1.2]{CC}).

We use the Poincar\'e disk model for the hyperbolic plane $\mathbb{H}^2$. In particular, the underling space of $\mathbb{H}^2$ is the open unit ball in $\mathbb R^2$ whose closure, denoted by $\overline{\mathbb{H}}^2$, is the unit disk. The boundary of $\overline{\mathbb{H}}^2$ is the unit circle $S^1$ and is called the boundary of $\mathbb{H}^2$ at infinity.  Given a point $p$ in $\mathbb{H}^2$ and a unit tangent vector $v\in UT_p\mathbb{H}^2$, there is a unique geodesic ray $\gamma_{p,v}: [0, \infty) \to \mathbb{H}^2$ for which $\gamma_{p,v}(0) = p$ and $\dot{\gamma}_{p,v}(0) = v$, and this geodesic ray limits to a unique point of $\partial \overline{\mathbb{H}}^2$.  This correspondence determines a canonical homeomorphism between $UT_p\mathbb{H}^2$ and $\partial \overline{\mathbb{H}}^2$ for any $p\in \mathbb{H}^2$. 

Since each leaf $\lambda$ of $\widetilde{\mathcal{F}}$ is isometric to the hyperbolic plane $\mathbb{H}^2$ with respect to the pull-back $\tilde g$ of $g$ to $\widetilde M$, each $\lambda$ gives rise to a circle at infinity which we denote by $\partial_\infty \lambda$. This association allows us to define two related $S^1$-bundles with fibres $\partial_\infty\lambda$. The first, denoted by $\bar{E}_\infty$, has base $\mathcal{L}$ and the second, denoted by $E_\infty$, has base $\widetilde M$. The topologies of these bundles are defined similarly. 

Let $\{(U_\alpha, \varphi_\alpha)\}$ be a regular foliated atlas of $\widetilde{\mathcal{F}}$ such that $\varphi_\alpha(U_\alpha)\cong R_\alpha \times B_\alpha$ where $R_\alpha$ is a rectangular region in $\mathbb{R}^2$ and $B_\alpha$ is an open interval in $\mathbb{R}$. We use $\mathcal F_\alpha$ to denote the foliation on $R_\alpha \times B_\alpha$ determined by the plaques $R_\alpha \times \{x\}$.

Recall that we have assumed that the transition maps $\varphi_\beta\circ \varphi_\alpha^{-1}$ are horizontally smooth. In particular, the differential ``$D(\varphi_\beta\circ \varphi_\alpha^{-1})$" is 
defined and varies continuously over $T\mathcal F_\alpha|_{\varphi_\alpha(U_\alpha \cap U_\beta)}$. Hence if $UT\widetilde{\mathcal{F}}$ and $UTR_\alpha$ denote the unit tangent bundles of $\widetilde{\mathcal F}$ and $R_\alpha$, the atlas $\{(U_\alpha, \varphi_\alpha)\}$ determines local trivialisations $\{\widetilde{\varphi}_\alpha\}$ of $UT\widetilde{\mathcal F}$: 
$$UT\widetilde{\mathcal F}|_{U_\alpha}\rightarrow (UTR_\alpha) \times B_\alpha \equiv (R_\alpha \times S^1) \times B_\alpha \equiv (R_\alpha \times B_\alpha) \times S^1 \cong U_\alpha \times S^1$$
whose transition functions $\widetilde{\varphi}_\beta \circ \widetilde{\varphi}_\alpha^{-1}: \varphi_\alpha(U_\alpha \cap U_\beta)  \times S^1 \to \varphi_\beta(U_\alpha \cap U_\beta)  \times S^1$ are continuous. 

Consider the fibre-preserving bijection 
\begin{equation} \label{eqn: G}
G: UT\widetilde{\mathcal{F}} \rightarrow E_\infty, (\tilde p, v) \mapsto \gamma_{(\tilde p, v)}
\end{equation} 
where $\gamma_{(\tilde p, v)}(t)$ is the geodesic ray on the leaf of $\widetilde{\mathcal F}$ containing $\tilde p$ which satisfies $\gamma_{(\tilde p, v)}(0)=\tilde{p}$ and $\dot{\gamma}_{(\tilde p, v)}(0)=v$. We use this bijection to topologise $E_\infty$ and endow it with the structure of a locally-trivial oriented $S^1$-bundle over $\widetilde M$ with transition functions $\{\widetilde{\varphi}_\beta \circ \widetilde{\varphi}_\alpha^{-1}\}$. 

Defining the topology on $\overline{E}_\infty$ is similar. Fix a transversal $\tau$ to $\widetilde{\mathcal F}$. The simple-connectivity of $\widetilde M$ implies that $\tau$ embeds in $\mathcal L$ with image $l$, say. As above, there is a bijection  
\begin{equation} \label{eqn: Gtau}
G_\tau: UT\widetilde{\mathcal{F}}|_\tau \rightarrow \bar E_\infty|_{l}, (\tilde p, v) \mapsto \gamma = \gamma_{(\tilde p, v)}
\end{equation} 
which we declare to be a homeomorphism. Distinct transversals with the same image in $\mathcal L$ determine the same topology on $\bar E_\infty|_{l}$ since the geometry of the leaves of $\widetilde{\mathcal F}$ varies continuously over compact subsets of $\widetilde M$. See \cite[\S 2.8]{CD} for more details and discussion.

\begin{remark} \label{remark: bundle action} 
By construction, the deck transformations of the cover $\widetilde M \to M$ determine isometries between the leaves of $\widetilde{\mathcal F}$ and as such, induce homeomorphisms between the fibres of $\bar{E}_\infty$ and $E_\infty$. The naturality of the topologies on $\bar{E}_\infty$ and $E_\infty$ is reflected in the fact these homeomorphisms determine actions of $\pi_1(M)$ on $\bar{E}_\infty$ and $E_\infty$ by bundle maps.  
\end{remark} 

\subsection{Circular orders and monotone maps}
\label{subsubsec:sections and universal circle}
A circular order on a set $O$ of cardinality $4$ or more is a collection of linear orders $<_p$ on $O \setminus \{p\}$, one for each $p \in O$, such that for $p, q \in O$, the linear orders $<_p$ and $<_q$ differ by a {\it cut} on $O \setminus \{p, q\}$. (See \cite[Definition 2.34]{Cal2} for the details.) If $O = \{x, y, z\}$ has three elements, we add the condition that $y <_x z$ if and only if $z <_y x$. Subsets of cardinality $3$ or more of circularly ordered sets inherit circular orders in the obvious way. 

The archetypal example of a circularly ordered set is an oriented circle where the linear orders $<_p$ on $S^1 \setminus \{p\}$ are those determined by the orientation. More generally, any subset of cardinality $3$ or more of an oriented circle inherits a circular order from the orientation on the circle.

Given a circularly ordered set $O$ of four or more elements, we define an ordered triple $(x, y, z) \in O^3$ to be {\it positively ordered} if there is a $p \in O \setminus \{x, y, z\}$ such that $x <_p y <_p z$. It is called {\it negatively ordered} if there is a $p \in O \setminus \{x, y, z\}$ such that $y <_p x <_p z$. We leave it to the reader to verify that a positively ordered triple is never negatively ordered and vice versa. Further, a triple of distinct points $(x, y, z)$ is positively ordered, respectively negatively ordered, if and only if $(y, z, x)$ is positively ordered, respectively negatively ordered. 
 
A totally ordered set $S$ admits a natural topology with basis consisting of the open interval $(x, y) = \{p \in S : x <_S p <_S y\}$. A map $f: S \to T$ between totally-ordered sets is called {\it monotone} if it is surjective and if $f^{-1}(t)$ is a closed interval $[x, y] = \{p \in S : x \leq_S p \leq_S y\}$ for each $t \in T$. Monotone maps are continuous and satisfy $f(s_1) \leq_T f(s_2)$ whenever $s_1 <_S s_2$. 

Analogous definitions are made for circularly ordered sets. Given distinct points $x, y$ in a circularly ordered set $O$, the {\it open interval} $(x, y)$ is $\{p \in O : (x, p, y) \hbox{ is positively ordered}\}$. Closed intervals are defined similarly. The complement of an open, respectively closed, interval is a closed, respectively open, interval. 

The set of open intervals in $O$ forms a basis of the {\it order topology} on $O$. Closed intervals are closed in this topology. If $O$ is a subset of an oriented circle with the induced circular order, then the order topology coincides with the subspace topology and the open intervals of $O$ are intersections of $O$ with open arcs of the circle. 

A map $f: O_1 \to O_2$ between circularly ordered sets is called {\it monotone} if it is surjective and point inverses are closed intervals. Then for any $p_2 \in O_2$ and $p_1 \in f^{-1}(p_2) \subset O_1$, the restriction of $f$ to $(O_1 \setminus f^{-1}(p_2), <_{p_1}) \to (O_2 \setminus \{p_2\}, <_{p_2})$ is a monotone map of totally ordered sets. Monotone maps are continous.

\subsection{Sections of $\bar{E}_\infty$ and universal circles} 
The key to the Calegari-Dunfield proof of the existence of a universal circle $S^1_{univ}$ of $\mathcal F$ is the construction of a certain set of sections $\mathcal{S} = \{\sigma: \mathcal{L}\rightarrow \bar{E}_\infty\}$ of the circle bundle $\bar{E}_\infty$ which is circularly orderable and is closed under the action of $\pi_1(M)$ on $\bar{E}_\infty$ (\cite[\S 6]{CD}). The set $\mathcal{S}$ is separable with respect to the order topology and contains no pair of distinct elements $\sigma_1, \sigma_2$ such that $(\sigma_1, \sigma_2) = \emptyset$ (i.e. $ \mathcal{S}$ has no {\it gaps}), so it can be embedded into an oriented circle as a dense ordered subspace. This circle turns out to be the universal circle $S^1_{univ}$. 

The action of $\pi_1(M)$ on $\mathcal{S}$ is order-preserving and continuous in the order topology, which implies that it extends to an orientation-preserving action on $S^1_{univ}$, yielding a homomorphism $\rho_{univ}: \pi_1(M)\rightarrow {\rm Homeo}_+(S_{univ}^1)$.

For each leaf $\lambda$ of $\mathcal{L}$, the evaluation map $e_\lambda: \mathcal{S} \rightarrow \partial_\infty\lambda$ is continuous and sends the closed interval $[\sigma_1, \sigma_2]$ into the interval $[\sigma_1(\lambda), \sigma_2(\lambda)]$. It extends to a monotone map $\phi_\lambda: S^1_{univ} \to \partial_\infty \lambda$ which, in particular, is continuous of degree one. 

It is shown in \cite{CD} that these objects satisfy the conditions of the following definition.

\begin{definition} {\rm (\cite[Definition 6.1]{CD})}
Let $\mathcal{F}$ be a co-oriented taut foliation of a closed oriented $3$-manifold $M$ and suppose that $M$ admits a Riemannian metric whose restriction to each leaf of $\mathcal F$ has constant curvature $-1$. A {\it universal circle} for $\mathcal{F}$ is a circle $S_{univ}^1$ together with the following data: 

$(1)$ A nontrivial representation 
 \begin{displaymath}
   \rho_{univ}: \pi_1(M)\rightarrow {\rm Homeo}_+(S_{univ}^1)
 \end{displaymath}

$(2)$ For every leaf $\lambda$ of the pull-back  foliation $\widetilde{\mathcal{F}}$, there is a monotone map 
 \begin{displaymath}
  \phi_\lambda: S_{univ}^1\rightarrow \partial_{\infty}\lambda
 \end{displaymath}

$(3)$ For every leaf $\lambda$ of $\widetilde{\mathcal{F}}$ and every $\alpha\in \pi_1(M)$, the following diagram commutes: 
\begin{figure}[ht]
  \centering
  \begin{tikzpicture}[scale=0.8]
 % \draw [help lines] (0,0) grid (12, 4);
   \draw [thick, ->] (2.8, 3) -- (5,3);
   \draw [thick, ->] (2.8, 1) -- (5,1);
   \draw [thick, ->] (2, 2.5) -- (2, 1.5);
   \draw [thick, ->] (5.4, 2.5) -- (5.4, 1.5);
   \node at (2,3) {$S_{univ}^1$};
   \node at (2,1) {$\partial_{\infty} \lambda$};
   \node [right] at (5,3) {$S_{univ}^1$}; 
   \node [right] at (5,1) {$\partial_{\infty} (\alpha\cdot \lambda)$};
   \node [right] at (5.5, 2) {$\phi_{\alpha\cdot \lambda}$};
   \node [left] at (2, 2) {$\phi_\lambda$};
   \node [above] at (3.9,3) {$\rho_{univ}(\alpha)$};
   \node [above] at (3.9,1) {$\alpha\cdot $};
  \end{tikzpicture}
\end{figure}

$(4)$ If $\lambda$ and $\mu$ are incomparable leaves of $\widetilde{\mathcal{F}}$, then the core of $\phi_\lambda$ is contained in the closure of a single gap of $\phi_\mu$ and vice versa.

\label{def:universal circle}
\end{definition}

We refer the reader to \cite{CD} for more details on condition (4), which will not play a role below. 

\begin{remark} \label{rem: rhouniv non-trivial} 
Suppose that $\mathcal F$ is a co-oriented taut foliation on $M$ which admits, as above, a universal circle. We claim that the image of $\rho_{univ}$ is an infinite group which is non-ableian if $M$ is a rational homology $3$-sphere. To see this, first note that $\mathcal F$ must have a non-simply connected leaf. Otherwise each of its leaves is isometric to $\mathbb H^2$ and in particular homeomorphic to $\mathbb R^2$, which implies that $M$ is the $3$-torus (\cite{Ros}, \cite{Ga2}). But this is impossible; transverse loops to $\mathcal{F}$ are homotopically non-trivial (see \cite[Theorem 4.35(3)]{Cal2}, for instance), so the assumption that each of the leaves of $\mathcal{F}$ is isometric to $\mathbb H^2$ implies that $\pi_1(M) \cong \mathbb Z^3$ would have exponential growth (\cite[Lemma 7.2]{Pla}), a contradiction. Thus $\mathcal F$ has a non-simply-connected leaf $\bar \lambda$. There is a leaf $\lambda$ of $\widetilde{\mathcal F}$ contained in the inverse image of $\bar \lambda$ which is invariant under the deck transformations corresponding to $\pi_1(\bar \lambda) \leq \pi_1(M)$. Since $\pi_1(\bar \lambda)$ acts on the hyperbolic plane $\lambda$ by isometries, it induces a faithful action of $\pi_1(\bar \lambda)$ on $\partial_\infty \lambda$. Hence as $\pi_1(\bar \lambda)$ is non-trivial, it is infinite. So by (3) of the definition of a universal circle, the image of $\rho_{univ}$ is an infinite group. If $M$ is a rational homology $3$-sphere, the image cannot be abelian as otherwise it would be finite. 
\end{remark}

\section{The Euler class of the universal circle action}
\label{sec:euler class universal circle}
Recall that $T\mathcal{F}$ denotes the oriented $2$-plane field over $M$ determined by $\mathcal{F}$. The goal of this section is to prove the following proposition.

\begin{proposition} {\rm (Thurston)}
Assume that $\mathcal{F}$ is a co-oriented taut foliation on a closed oriented $3$-manifold $M$ and that there is a Riemannian metric $g$
 on $M$ which restricts to a metric of constant curvature $-1$ on each leaf $\lambda$ of $\mathcal{F}$. Let $\rho_{univ}:\pi_1(M)\rightarrow {\rm Homeo}_+(S_{univ}^1)$ be a universal circle representation associated to $(M,g,\mathcal{F})$. Then the Euler class of the oriented circle bundle $E_{\rho_{univ}}$ equals that of $T\mathcal{F}$.
 \label{prop:Euler class}
\end{proposition}

\begin{remark}
We know of no proof of this fundamental result in the literature. We were made aware of it on separate occasions by Ian Agol, Danny Calegari, and Nathan Dunfield. \end{remark}

\begin{proof}
As above, $(\widetilde M, \tilde g,\widetilde{\mathcal{F}})$ denotes the universal cover of $M$ equipped with the pull-back foliation $\widetilde{\mathcal{F}}$ and the pull-back  metric $\tilde g$.

Let $\Phi: \widetilde M \times S^1_{univ} \to E_\infty$ be the fibre-preserving map sending $(\tilde{p},\sigma)$ to $\phi_\lambda(\sigma)$, where $\lambda$ is the leaf of $\mathcal F$ containing $\tilde p$ and $\phi_\lambda:S_{univ}\rightarrow \partial_\infty \lambda$ is the degree one monotone map of Definition \ref{def:universal circle}(2). The continuity of $\Phi$ will be verified in Lemma \ref{lem: continuity of Phi} and we assume it for now. 

Recall the bundle isomorphism $G:UT\widetilde{\mathcal{F}}\rightarrow E_\infty$ with $(\tilde{p},v)\mapsto \gamma_{(\tilde{p},v)}$ defined in \S\ref{subsubsec:circle bundles at infinity}. By composing $\Phi$ with $G^{-1}$, we obtain a fibre-preserving map 
 \begin{displaymath}
\widetilde F:=G^{-1}\circ \Phi: \widetilde M\times S_{univ}^1\rightarrow UT\widetilde{\mathcal{F}} 
 \end{displaymath}
which restricts to a degree one monotone map between fibres. 

We claim that $\widetilde F$ descends to a fibre-preserving map $F: E_{\rho_{univ}} \rightarrow UT\mathcal{F}$. To see this, note that $\pi_1(M)$ acts on both $\widetilde{M}\times S^1_{univ}$ and $UT\widetilde{\mathcal{F}}$ with quotient spaces  $E_{\rho_{univ}}$ and $UT\mathcal{F}$ respectively. The existence of $F$ will follow if we can show that $h\widetilde F=\widetilde Fh$ for any $h \in \pi_1(M)$.  But given $\tilde p\in \lambda$ on $\widetilde{M}$ and $\sigma\in S^1_{univ}$, we have $$\widetilde F(h\cdot(\tilde p, \sigma)) =\widetilde F(h\tilde{p}, \rho_{univ}(h)\sigma)= G^{-1}\circ\phi_{h\lambda}(\rho_{univ}(h)\sigma).$$
By Definition \ref{def:universal circle}(3),  we have $\phi_{h\lambda}(\rho_{univ}(h)h\sigma)=h\phi_\lambda(\sigma)$. Hence 
\begin{align*}
\widetilde F(h\cdot(\tilde p, \sigma))=G^{-1}\circ h\phi_\lambda(\sigma).  
\end{align*}
Since $\pi_1(M)$ acts on $(\widetilde{M},\tilde{g})$ by isometries, $$\widetilde F(h\cdot(\tilde p, \sigma))=G^{-1}\circ h\phi_\lambda(\sigma)=hG^{-1}\circ \phi_\lambda=h\widetilde F(\tilde{p},\sigma),$$
which is what we needed to show.

Let $D_{\rho_{univ}}$ and $DT\mathcal{F}$ be the oriented disk bundles associated to $E_{\rho_{univ}}$ and $UT\mathcal{F}$ respectively (cf. \S \ref{sec: euler class}), and let $F_D: D_{\rho_{univ}}\rightarrow DT\mathcal{F}$ denote the map induced by $F$. We have the following commutative diagram in which $D_\infty$ denotes the fibre of $D_{\rho_{univ}}$ at $p\in M$:  
\begin{figure}[ht]
  \centering
  \begin{tikzcd}
(D_{\infty}, S_{univ}^1) \arrow[hookrightarrow]{r} \arrow{d}{(F_D,F)|_p}
  & (D_{\rho_{univ}}, E_{\rho_{univ}}) \arrow{d}{(F_D,F)}
  \\
(DT_p\mathcal{F}, UT_p\mathcal{F}) \arrow[hookrightarrow]{r}
  & (DT\mathcal{F}, UT\mathcal{F})
\end{tikzcd}
\end{figure}  
\vspace{-.3cm} 

Since $F$ restricts to a degree one map between fibres, $(F_D,F)|_p^*$ is an isomorphism which sends the orientation class in $H^2(DT_p\mathcal{F}, UT_p\mathcal{F})$  to the orientation class  in $H^2(D_{\infty}, S_{univ}^1)$ . Hence  $(F_D,F)^*$ sends the Thom class of $DT\mathcal{F}$ in $H^2(DT\mathcal{F}, UT\mathcal{F})$ to the Thom class of $D_{\rho_{univ}}$ in $H^2(D_{\rho_{univ}}, E_{\rho_{univ}})$. By definition, then, we have $e(T\mathcal{F})=e(UT\mathcal{F})=e(E_{\rho_{univ}})$ (see \S \ref{sec: euler class}). 
\end{proof}

To complete the proof, it remains to prove that $\Phi$ is continuous. 

\begin{lemma} \label{lem: continuity of Phi}
The map $\Phi: \widetilde M \times S^1_{univ} \to E_\infty$ is continuous. 
\end{lemma}

\begin{proof}
It suffices to show that for any foliation chart $(U_\alpha, \varphi_\alpha)$ of $\widetilde{\mathcal{F}}$, the restriction $\Phi|_{U_\alpha}: U_\alpha\times S^1_{univ}\rightarrow E_\infty|_{U_\alpha}$ is continuous. (See \S \ref{subsubsec:circle bundles at infinity}.)

Let $l_\alpha$ be the open interval on $\mathcal{L}$ corresponding to a transversal in $U_\alpha$ and $e_{l_\alpha}: \mathcal{S}\rightarrow \mathcal{S}|_{l_\alpha}$ the map which restricts a section in $\mathcal{S}$ to $l_\alpha$. Now define $\mathcal{S}_{l_\alpha}$ to be the image of $e_{l_\alpha}$. That is, 
\begin{displaymath}
 \mathcal{S}_{l_\alpha} = \{\sigma|_{l_\alpha} \, : \, \sigma \in \mathcal{S}\}
\end{displaymath}
The inverse image by $e_{l_\alpha}$ of an element of $\mathcal{S}_{l_\alpha}$ is a closed interval in $\mathcal{S}$, since sections in $\mathcal{S}$ do not cross each other (\cite[\S 6.14]{CD}). Hence, the circular order on $\mathcal{S}$ defines a circular order on $\mathcal{S}_{l_\alpha}$ and if we equip $\mathcal{S}_{l_\alpha}$ with the associated order topology, then $e_{l_\alpha}$ is a monotone map between two circularly ordered sets. In particular, $e_{l_\alpha}$ is continuous. 

On the other hand, $\mathcal{S}_{l_\alpha}$ is a subset of the set of continuous functions $C^0(l_\alpha, \bar{E}|_{l_\alpha})$ from $l_\alpha$ to $\bar{E}|_{l_\alpha}$. One can check that the order topology on $\mathcal{S}_{l_\alpha}$ agrees with the subspace topology induced by the compact-open topology on $C^0(l_\alpha, \bar{E}|_{l_\alpha})$. We denote the closure of $\mathcal{S}_{l_\alpha}$ in $C^0(l_\alpha, \bar{E}|_{l_\alpha})$ by $\bar{\mathcal{S}}_{l_\alpha}$. 

Note that for any leaf $\lambda\in l_\alpha$, the evaluation map $e_\lambda: \mathcal{S}\rightarrow \partial_\infty \lambda$ factors through $\mathcal{S}_{l_\alpha}$. That is, the left-hand diagram immediately below commutes and its maps extend by continuity to yield the right-hand diagram. 

\begin{center} 
\begin{tikzpicture}[scale=0.8]
\node at (8, 4.5) {\small $U_\alpha \times \mathcal{S}$};
\node at (13, 4.5) {\small $E_\infty|_{U_\alpha}$};
\draw [thick, ->] (9, 4.5) --(12.1,4.5); 
\node at (10.5, 4.9) {\small $\Phi$}; 
\node at (8, 1.5) {\small $U_\alpha \times \mathcal{S}_{l_\alpha}$};
\draw [->, thick] (9,2) --(12.5, 4);
\node [right] at (10.8, 2.7) {\small $\Phi_{l_\alpha}$};
\draw [thick, ->] (8, 4) -- (8,2);
\node at (7, 3) {\small $1_{u_\alpha} \times e_{l_\alpha}$};
\end{tikzpicture}
\hspace{1cm} 
\begin{tikzpicture}[scale=0.8]
\node at (8, 4.5) {\small $U_\alpha \times S^1_{univ}$};
\node at (13, 4.5) {\small $E_\infty|_{U_\alpha}$};
\draw [thick, ->] (9, 4.5) --(12.1,4.5); 
\node at (10.5, 4.9) {\small$\Phi$}; 
\node at (8, 1.5) {\small $U_\alpha \times \overline{\mathcal{S}}_{l_\alpha}$};
\draw [->, thick] (9,2) --(12.5, 4);
\node [right] at (10.8, 2.7) {\small $\Phi_{l_\alpha}$};
\draw [thick, ->] (8, 4) -- (8,2);
\node at (7, 3) {\small $1_{u_\alpha} \times \overline{e}_{l_\alpha}$};
\end{tikzpicture}
\end{center}

Since the evaluation map from $C^0(l_\alpha, \bar{E}|_{l_\alpha}) \times l_\alpha \to \bar{E}|_{l_\alpha}$ is continuous with respect to the compact open topology, it follows that $
\Phi_{l_\alpha}: U_\alpha\times \bar{\mathcal{S}}_{l_\alpha}\rightarrow E_\infty|_{U_\alpha}$ is continous. Therefore, $\Phi$ is continuous over $U_\alpha\times S^1_{univ}$.  
\end{proof}

\section{Left-orderability of $3$-manifold groups and universal circles}
\label{sec: lo and taut foliations}
 
A group $G$ is said to be {\it left-orderable} if it is nontrivial and there exists a strict total order $<$ on $G$ such that if $a, b, c \in G$ and $a<b$, then $ca<cb$. 

The group ${\rm Homeo}_+(\mathbb{R})$ is left-orderable (see, for instance, the proof of \cite[Theorem 6.8]{Ghy}) and serves as a universal host for countable left-orderable groups. Indeed, a countable group $G \ne \{1\}$ is left-orderable if and only if it admits a faithful representation into ${\rm Homeo}_+(\mathbb{R})$ (cf. \cite[Theorem 6.8]{Ghy})). If $G$ is the fundamental group of an orientable irreducible $3$-manifold, the condition that the representation be faithful can be removed. 

\begin{theorem} {\rm (\cite[Theorem 1.1]{BRW})}   \label{thm:brw}
Assume that $M$ is a compact, orientable, irreducible $3$-manifold. Then $\pi_1(M)$ is left-orderable if and only if it admits a homomorphism to ${\rm Homeo}_+(\mathbb{R})$ with non-trivial image. Equivalently, $\pi_1(M)$ is left-orderable if and only if it admits a left-orderable quotient. 
\end{theorem}

Consequently, 

\begin{corollary}[\cite{HS, BRW}] 
\label{cor:b1 is not zero lo}
 Let $M$ be a compact, orientable, prime $3$-manifold and let $b_1(M)$ denote its first Betti number. If $b_1(M) > 0$, then $\pi_1(M)$ is left-orderable. 
\qed 
\end{corollary}

The following theorem states a known criterion for the left-orderability of the fundamental group of a rational homology $3$-sphere.  

\begin{theorem}  \label{prop:taut foliation left orderability}
Let $M$ be a rational homology $3$-sphere which admits a co-orientable taut foliation whose tangent plane field has zero Euler class. Then $\pi_1(M)$ is left-orderable.
\end{theorem}

\begin{proof}
First we observe that rational homology $3$-spheres which admit co-orientable taut foliations are irreducible \cite{Nov}. Hence $\pi_1(M)$ will be left-orderable if it admits a homomorphism to ${\rm Homeo}_+(\mathbb{R})$ with non-trivial image (Theorem \ref{thm:brw}). 

Fix a Riemannian metric $g$ on $M$. Since $\mathcal{F}$ is orientable, any leaf of $\mathcal F$ which is not conformally negatively curved with respect to the induced metric gives rise to a nontrivial homology class in $H_2(M; \mathbb{R})$ (\cite[Corollary 6.4]{Pla}), contrary to the fact that $M$ is a rational homology sphere. Consequently, each leaf of $\mathcal{F}$ is conformally hyperbolic. Then by \cite[Theorem 4.1]{Can}, $g$ is conformal to a metric $g'$ whose restriction to each leaf has constant curvature $-1$. Hence there exists a universal circle action $\rho:\pi_1(M)\rightarrow {\rm Homeo}_+(S^1)$ which is non-trivial by Remark \ref{rem: rhouniv non-trivial}. By Proposition \ref{prop:Euler class}, we have $e(E_\rho)=e(T\mathcal{F})$, which is zero by hypothesis. Hence $\rho$ lifts to a non-trivial action of $\pi_1(M)$ on the real line (Lemma \ref{lem:milnor Euler class}) and therefore by Theorem \ref{thm:brw}, the fundamental group $\pi_1(M)$ is left-orderable.
\end{proof}

\section{The left-orderability of the fundamental groups of cyclic branched covers of fibred knots}
\label{sec: fdtc and lo} 

In this section we consider cyclic branched covers of hyperbolic fibred knots and prove Theorem \ref{thm:conjecture fibre knots},  Corollary \ref{cor:universal abelian cover} and Corollary \ref{cor: knots rational homology sphere}. 

Recall that a $3$-manifold is called excellent if it is not an L-space, admits a co-orientable taut foliation and its fundamental group is left-orderable. 

\begin{proposition}\label{prop: e=0 and c>1 implies lo}
 Let $M$ be an oriented rational homology sphere admitting an open book $(S,h)$ with binding a knot $K$ and pseudo-Anosov monodromy $h$. Let $\mathcal{F}_0$ denote the foliation on the exterior of $K$ given by the locally-trivial fibre bundle structure. Suppose that $e(T\mathcal{F}_0) = 0$. 
If $|c(h)| \geq 1$, then $M$ is excellent.
\end{proposition}

\begin{proof}
If $c(h)\leq -1$, we can consider the open book decomposition $(-S, h^{-1})$ of $M$. By Lemma \ref{lem:poincare translation number}, $c(h^{-1})=-c(h)\geq 1$.  Hence, we may assume that $c(h)\geq 1$.  

By Theorem \ref{thm:cgeq1}, $M$ admits a co-oriented taut foliation $\mathcal F$ whose tangent plane field is homotopic to the contact structure $\xi$ supported by $(S,h)$. In particular,  the restriction of $\xi$ to the knot complement $X(K)$ is homotopic to $\mathcal{F}_0$. It follows that the Euler class $e(\xi)$ is sent to $0$ under the inclusion-induced homomorphism $H^2(M) \to H^2(X(K))$. By Lemma \ref{lemma: homology of exterior}, this homomorphism is an isomorphism and hence, $e(T\mathcal{F})=e(\xi)= 0$. This implies that $M$ has a left-orderable fundamental group by Theorem \ref{prop:taut foliation left orderability}. 
\end{proof}

\begin{proof}[Proof of Theorem \ref{thm:conjecture fibre knots}]
Let $K$ be a hyperbolic fibred knot in an oriented integer homology $3$-sphere $M$ with fibre $S$ and monodromy $h$. Then $X_n(K)(\mu_n + q \lambda_n)$ has open book decomposition $(S, T_{\partial}^{-q} h^n)$ with binding the core of the filling solid torus, which we denote by $\widetilde K$. The exterior of $\widetilde K$ in $X_n(K)(\mu_n + q \lambda_n)$ is $X_n(K)$. 

Since $T_\partial^{-q}h^n$ is freely isotopic to $h^n$, there is a fibre-preserving homeomoprhism between the mapping tori of $T_\partial^{-q}h^n$ and $h^n$. Hence the foliation on $X_n(K)$ determined by the open book $(S, h^n)$, denoted by $\widetilde{\mathcal{F}}_0$, is isomorphic to the one determined by the open book $(S,T_\partial^{-q}h^n)$ and the same holds for their tangent plane fields. We show that $e(T\widetilde{\mathcal{F}}_0)=0$.

Since $M$ is an integer homology $3$-sphere, $H^2(X(K)) \cong 0$, so that if $\mathcal{F}_0$ is the foliation of $X(K)$ determined by the open book $(S, h)$, then $e(T\mathcal{F}_0) = 0$. Since the Euler class $e(T\widetilde{\mathcal{F}}_0)$ is the image of $e(\mathcal{F}_0)$ under the homomorphism $H^2(X(K))\to H^2(X_n(K))$, it is also zero.

By Lemma \ref{lem:poincare translation number}, we have  $|c(T_\partial^{-q} h^n)|=|nc(h)-q|\geq 1$. Hence by Proposition \ref{prop: e=0 and c>1 implies lo}, if $X_n(K)(\mu_n + q \lambda_n)$ is a rational homology sphere, it is excellent. Otherwise, the first Betti number of $X_n(K)(\mu_n + q \lambda_n)$ is positive and it is also excellent (\cite{BRW, Ga1}). This proves part (2) of Theorem \ref{thm:conjecture fibre knots}. Part $(1)$ is an immediate consequence of part (2) and Proposition \ref{prop:lower bound FDTC}.
\end{proof}

\begin{proof}[Proof of Corollary \ref{cor:universal abelian cover}]
Assume that $n, q$, and $h$ are given as in the statement of the corollary and that $q \not \in \left\{ 
\begin{array}{cl} 
\{nc(h)\}  & \hbox{ if } nc(h) \in \mathbb Z \\ 
\{\lfloor nc(h) \rfloor, \lfloor nc(h) \rfloor+ 1\}  & \hbox{ if } nc(h) \not  \in \mathbb Z  
\end{array} \right.$. We must show that the $n$-fold cyclic cover of $X(K)(n\mu+q\lambda)$ is excellent. Since this cover is homeomorphic to the Dehn filling $X_n(K)(\mu_n+q\lambda_n)$ and the latter is excellent if $|nc(h)-q|\geq 1$ by Theorem \ref{thm:conjecture fibre knots}(2), we need only verify that this inequality holds.
But $|nc(h)-q|<1$ if and only if either $q = n c(h) \in \mathbb Z$ or $nc(h) \notin \mathbb{Z}$ and $q$ is either 
$\lfloor nc(h) \rfloor$ or $ \lfloor nc(h) \rfloor + 1$. As each of these cases has been excluded, the corollary holds. 
\end{proof}

\begin{proof}[Proof of Corollary \ref{cor: knots rational homology sphere}] 
Under the assumption, the $n$-fold cyclic branched cover of $X(K)(p\mu+q\lambda)$ is homeomorphic to $X_n(K)(\mu_n+mq\lambda_n)$, where $n=mp$ and $(p,q)=1$. The claim therefore follows from Theorem \ref{thm:conjecture fibre knots}(2). 
\end{proof}

\section{The left-orderability of the fundamental groups of cyclic branched covers of closed braids}
\label{sec:cyclic closed braids}
In this section we examine Conjecture \ref{conj: lspace} for cyclic branched covers of hyperbolic closed braids.

\subsection{Open book decomposition of cyclic branched covers of closed braids}
\label{subsec:lift open book}
Given a punctured surface $S$, we let $\bar S$ denote the compact surface obtained from $S$ by filling its punctures. 
If  $f: S\rightarrow S'$ is a proper continuous map between two punctured surfaces, we use $\bar f: \bar S\rightarrow \bar S'$ to denote its continuous extension. 

Given a $m$-braid $b: D_m\rightarrow D_m$, the pair $(\bar{D}_m,\bar{b})$ defines an open book decomposition of  $S^3$ with pages diffeomorphic to the interior of the unit disk $\bar D_m$ and the monodromy $\bar{b} : \bar{D}_m \rightarrow \bar{D}_m$ is the extension of the mapping class $b$ to $\bar{D}_m$. This open book decomposition of $S^3$ lifts to an open book decomposition of the $n$-fold cyclic branched cover of $S^3$ along the closed braid $\hat{b}$, as we describe now.

Let $\mathfrak{p}:\Sm \rightarrow D_m$ be the $n$-fold cyclic cover of the $m$-punctured disk associated with the epimorphism
$\pi_1(D_m; p)\rightarrow \mathbb{Z}/n$ which maps each 
generator $x_i$ to the class of $1$ in $\mathbb{Z}/n$. See Figure \ref{fig:punctured disk generators basepoint}.

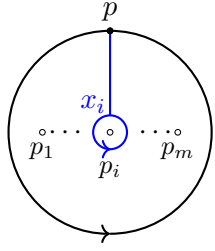
\begin{figure}[ht]
 \centering
 \begin{tikzpicture}[scale=0.75]
 % \draw [help lines] (0,0) grid (5,5);
  \draw (1.3,2.5) circle (0.05);
  \node [below] at (1.3, 2.5) {\small $p_1$};
  \node at (1.75, 2.5) {$\cdots$};
  \draw (2.5,2.5) circle (0.05); 
  \draw [blue, thick] (2.5,2.5) circle (0.3); 
  \draw [blue, ->, thick] (2.49,2.2) -- (2.51,2.2);
  \draw [blue, thick] (2.5, 4.3) -- (2.5, 2.8);
   \node [below] at (2.5,2.16) {\small $p_i$};
  \node at (3.35, 2.5) {$\cdots$};
  \draw (3.7,2.5) circle (0.05);
  \node [below] at (3.7,2.5) {\small $p_m$};
% draw the generator
 % \draw [blue, thick, ->] (2.5, 4.3) to  [in=180,out=255] (2.5,1.9);
 % \draw [blue, thick] (2.5,1.9) to [in=285, out=0] (2.5,4.3);   % basepoint
   \node [blue] at (2.2,3) {$x_i$};
  \filldraw (2.5,4.3) circle (0.05);
  \node [above] at (2.5, 4.3) {$p$};
  \draw [thick, ->] (2.5, 4.3) to  [in=90,out=180] (0.7,2.5) to [in=180,out=270] (2.5,0.7);
  \draw [thick] (2.5,0.7) to [in=270, out=0] (4.3,2.5) to [in=0, out=90] (2.5,4.3);   % basepoint
  \filldraw (2.5,4.3) circle (0.05);
  %\draw [thick] (2.5,2.5) circle (1.8);
 \end{tikzpicture}
  \caption{Generators $x_i$ of $\pi_1(D_m)$.}
  \label{fig:punctured disk generators basepoint}
\end{figure}

The automorphism $b_*: \pi_1(D_m;p)\rightarrow \pi_1(D_m;p)$ induced by $b$ sends each generator $x_i$ to a conjugate of a generator $x_j$ for some $j$. Consequently, $b_*$ restricts to an automorphism of the kernel of  $\pi_1(D_m;p)\rightarrow \mathbb{Z}/n$ and therefore $b: D_m\rightarrow D_m$ lifts to a diffeomorphism $\psi: \Sm \rightarrow \Sm$ such that $\psi|_{\partial \Sm}$ is the identity and $b \circ \mathfrak{p} = \mathfrak{p} \circ \psi$. Hence $\bar b \circ \bar{\mathfrak{p}} = \bar{\mathfrak{p}} \circ \bar \psi$ where $\bar{\mathfrak{p}}: \bSm\rightarrow \bar{D}_m$ is an $n$-fold cyclic branched cover of the disk $\bar D_m$ along $m$ points.

It is routine to check that the pair $(\bSm,\bar{\psi})$ defines an open book decomposition of  the $n$-fold cyclic branched cover $\Sigma_n(\hat{b})$ of $S^3$ along the closed braid $\hat{b}$. 

\begin{lemma}
 If $b:D_m\rightarrow D_m$ is a pseudo-Anosov braid, then $\bar{\psi}: \bSm\rightarrow \bSm$ is also pseudo-Anosov.
 \label{lem:lift pseudo Anosov map}
\end{lemma}
\begin{proof}
Let $h_t:D_m\rightarrow D_m$ be a free isotopy from $b$ to its pseudo-Anosov representative, denoted by $\beta$, whose stable and unstable measured singular foliations  are denoted by $(\mathcal{F}^s,\mu^s)$ and $(\mathcal{F}^u,\mu^u)$. The isotopy $h_t$ lifts to an isotopy $H_t:\Sm\rightarrow \Sm$ with 
$H_0=\psi$, where $\psi:\Sm\rightarrow \Sm$ is the lift of the braid $b$ as defined above. 

By construction, $H_t$ leaves each component of $\partial S_n$ invariant. 

Set $\varphi :=H_1: \Sm\rightarrow \Sm$ and note that $\varphi$ is a pseudo-Anosov homeomorphism whose stable and unstable singular foliations are lifts of singular foliations $(\mathcal{F}^s,\mu^s)$ and $(\mathcal{F}^u,\mu^u)$, which we denote by $(\widetilde{\mathcal{F}}^s,\tilde\mu^s)$ and $(\widetilde{\mathcal{F}}^u,\tilde\mu^u)$ respectively. Note that under this lift, any $1$-pronged puncture on $D_m$ is lifted to an $n$-pronged puncture on $\Sm$. 

\medskip
 
 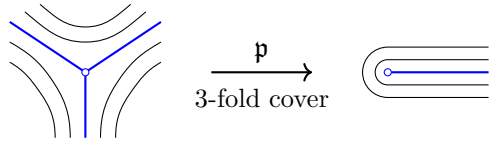
\begin{figure}[ht]
   \centering
   \begin{tikzpicture}[scale=0.67]
    %\draw [help lines] (0,0) grid (12, 5);
    \draw [blue, thick] (3,2.5) -- (1.5, 3.5); 
    \draw [blue, thick] (3,2.5) -- (4.5, 3.5);
    \draw [blue, thick] (3,2.5) -- (3,1.2);
    \filldraw [white] (3, 2.5) circle (0.06);
    \draw [blue] (3, 2.5) circle (0.07);
    \draw (2.7, 1.2) to [out=90, in=305] (2.3,2.4) to [out=125, in=330] (1.5, 3.1);
    \draw (3.3, 1.2) to [out=90, in=235] (3.7,2.4) to [out=55, in=210] (4.5, 3.1);
    \draw (2.4, 1.2) to [out=90, in=305] (2.1,2.1) to [out=125, in=330] (1.5, 2.7);
    \draw (3.6, 1.2) to [out=90, in=235] (3.9,2.1) to [out=55, in=210] (4.5, 2.7);
    \draw (1.8,3.7) to [out=320, in=180] (3, 3.1) to [out=0, in=220] (4.2,3.7);
    \draw (2.1,3.85) to [out=320, in=180] (3, 3.4) to [out=0, in=220] (3.9,3.85);
    % The arrow 
    \draw [thick,->] (5.5, 2.5) -- (7.5, 2.5);
    \node [above] at (6.5,2.5) {$\mathfrak{p}$};
    \node at (6.5,2) {\small $3$-fold cover}; 
    % Right portion
    \draw [thick, blue] (9,2.5) -- (11, 2.5);
    \filldraw [white] (9, 2.5) circle (0.064);
    \draw [blue] (9, 2.5) circle (0.07);
    \draw (11,2.75) to (9,2.75) to [out=180, in=90] (8.75,2.5) to [out=270,in=180] (9, 2.25) to (11,2.25);
    \draw (11,3) to (9,3) to [out=180, in=90] (8.5,2.5) to [out=270,in=180] (9, 2) to (11,2);
   \end{tikzpicture}
    \caption{A $1$-pronged singular point in $D_m$ is covered by an $n$-pronged singular point in $\Sm$}
  \label{fig:singularity}
 \end{figure}

Hence $\widetilde{\mathcal{F}}^s$ and $\widetilde{\mathcal{F}}^u$ extend to a transverse pair of measured singular foliations on the branched cover $\bSm$ which are invariant under  $\bar\varphi: \bSm\rightarrow \bSm$. Moreover, the extension $\bar H_t: \bSm\rightarrow \bSm$ of the isotopy $H_t:\Sm\rightarrow \Sm$ defines a free isotopy between $\bar \psi$ and the pseudo-Anosov homeomorphism $\bar\varphi$.  By definition, this shows that $\bar\psi$ is pseudo-Anosov.
\end{proof}

It is easy to check that the boundary of $\Sm$ has $(m,n)$ components. In particular, $\partial \Sm$ is not connected when $m$ and $n$ are not coprime. On the other hand, since the isotopy $H_t$ in the proof of Lemma \ref{lem:lift pseudo Anosov map} is equivariant with respect to the deck transformations of $\mathfrak{p}: S_n\rightarrow D_m$, the fractional Dehn twist coefficient of $\psi$ with respect to any two boundary components of $S_n$ are equal. We denote this number by $c(\psi)$. Similarly, the fractional Dehn twist coefficients of $\bar\psi$ are equal with respect to all boundary components of the branched cover $\bar{S}_n$, which we denote by $c(\bar\psi)$.

\begin{lemma}
 $c(\bar \psi) = c(\psi) = \frac{(m,n)}{n}c(b)$. 
 \label{lem:FDTC of lifting monodromy}
\end{lemma}

\begin{proof}
We continue to use the notation developed in the proof of Lemma \ref{lem:lift pseudo Anosov map}. 

First of all, $c(\bar\psi)=c(\psi)$ since the two isotopies $H_t$ and $\bar H_t$ are identical over the collar neighborhoods of $\partial\Sm$ and $\partial \bSm$. It remains to show that $c(\psi)=\frac{(m,n)}{n}c(b)$. 

Assume first that  $(m,n)=1$. In this case, $\partial \Sm$ is connected, and we denote it by $C$, so the restriction $\mathfrak{p}|_C: C \rightarrow \partial D_m$ is an $n$-fold cyclic cover. The proof that $c(\psi)=\frac{(m,n)}{n}c(b)$ is essentially contained in Figure \ref{fig:fractional Dehn twist lifting map}. To write it down more precisely, we need some notation.

Let $\{p_0,\cdots, p_{N-1}\}$ be the set of singular points  on $\partial D_m$ of the stable foliation $\mathcal{F}^s$. Fix a preimage $q_0$ of $p_0$ on $C=\partial S_n$. For $k=0, \cdots, n-1$, let $q^k_i \in C$ be the $k^{th}$ lift of the singular point $p_i$. That is, if $\gamma_{p_0p_k}$ is the subarc of $\partial D_m$ with endpoints $p_0$ and $p_k$ (cf. \S \ref{subsec:FDTC isotopy}), then $q^k_i$ is the end point of the unique lift of the path $C^k\cdot \gamma_{p_0p_k}$ starting at $q_0$. To simplify notation, we denote $q^0_i$ by $q_i$. In particular, $q_0^0=q_0$.  Note that $\{q^k_i\}$ is the set of singular points on $\partial \Sm=C$ of the stable foliation $\widetilde{\mathcal{F}}^s$ of $\varphi$. 
\begin{figure}[ht]
\centering
\begin{tikzpicture}[scale=0.9]
% bottom 
\node at (3.5, 0) {\footnotesize $\partial \Sm\times [0,1]$};
\node at (10.5, 0) {\footnotesize $\partial D_m\times [0,1]$};
% arrows 
 %\draw [help lines] (0,0) grid (14, 5);
 \draw [thick, ->] (6, 2.5) -- (8,2.5);
 \node [above] at (7, 2.5) {\footnotesize $\mathfrak{p}|_C$ is a};
 \node [below] at (7,2.5) {\footnotesize $3$-fold cover};
 % left circle
 \filldraw [light-gray] (3.5,2.5) circle (1.5);  
 \filldraw [white] (3.5,2.5) circle (0.8);
 \draw (3.5,2.5) circle (1.5);  
 \draw  (3.5,2.5) circle (0.8);
  \filldraw [light-gray] (10.5,2.5) circle (1.5);  
 \filldraw [white] (10.5,2.5) circle (0.8);
 \draw (10.5,2.5) circle (1.5);  
 \draw  (10.5,2.5) circle (0.8);
  % curve 
 \begin{scope}[decoration={
    markings,
    mark=at position 0.5 with {\arrow{>}}}
    ] 
 \draw [blue,postaction={decorate}] (10.5, 1.7) to [out=-10, in=270] (11.45, 2.5) to [out=90,in=0] (10.5, 3.5) to [out=180, in=90] (9.5,2.5) to [out=270, in=180] (10.5, 1.4) to [out=0, in=250] (12, 2.5);
 \node at (11.6, 3.2) {\color{blue} \tiny $h_t|_{p_0}$};
 \draw [blue, postaction={decorate}] (3.52, 1.67) to [out=-30, in=230] (4.5,2) to [out=50, in=280] (4.7,3) to [out=100, in=310] (4.25, 3.8);
 \node at (4, 1.4) {\color{blue}\tiny $H_t|_{q_0}$};
 \end{scope}
% makes on the right circle on the right
 \foreach \a in {0}
 {
\draw  [yshift=2.5cm, xshift=3.5cm]  (\a*360/4-90: 0.6) node {\tiny $q_{\a}$};
 \filldraw  [yshift=2.5cm, xshift=3.5cm]  (\a*360/4-90: 0.8) circle (0.03);
 }
 \foreach \a in {0,1,2,3}
 {
\draw  [yshift=2.5cm, xshift=10.5cm]  (\a*360/4-90: 1.75) node {\tiny $p_{\a}$};
 \filldraw  [yshift=2.5cm, xshift=10.5cm]  (\a*360/4-90: 1.5) circle (0.03);
 }

% markers on the left circle 
 \foreach \a in {0}
 {
\draw  [yshift=2.5cm, xshift=10.5cm]  (\a*360/4-90: 0.6) node {\tiny $p_{\a}$};
 \filldraw  [yshift=2.5cm, xshift=10.5cm]  (\a*360/4-90: 0.8) circle (0.03);
 }

  \foreach \a in {0,1,2,3}
 {
\draw  [yshift=2.5cm, xshift=3.5cm]  (\a*120/4-90: 1.75) node {\tiny $q_{\a}$};
 \filldraw  [yshift=2.5cm, xshift=3.5cm]  (\a*120/4-90: 1.5) circle (0.03);
 }
   \foreach \a in {0,1,2,3}
 {
\draw  [yshift=2.5cm, xshift=3.5cm]  (\a*120/4+30: 1.75) node {\tiny $q^1_{\a}$};
 \filldraw  [yshift=2.5cm, xshift=3.5cm]  (\a*120/4+30: 1.5) circle (0.03);
 }
   \foreach \a in {0,1,2,3}
 {
\draw  [yshift=2.5cm, xshift=3.5cm]  (\a*120/4+150: 1.75) node {\tiny $q^2_{\a}$};
 \filldraw  [yshift=2.5cm, xshift=3.5cm]  (\a*120/4+150: 1.5) circle (0.03);
 }
%   \foreach \a in {0,1,2,3}
% {
%\draw  [yshift=2.5cm, xshift=3.5cm]  (\a*90/4+180: 1.75) node {\tiny $q^3_{\a}$};
% \filldraw  [yshift=2.5cm, xshift=3.5cm]  (\a*90/4+180: 1.5) circle (0.03);
% }
 % arrows on the circle 
 \draw [->] (4.3,2.45)--(4.3,2.5);
 \draw [->] (11.3,2.45)--(11.3,2.5);
 \end{tikzpicture}
\caption{In the figure, $\{q^k_0, \cdots, q^k_3\}_{k=0,1,2}$ and $\{p_0, \cdots, p_3\}$ are singular points of the stable foliations $\widetilde{\mathcal{F}}^s$ and  $\mathcal{F}^s$ respectively. The path $H_t|_{q_0}$ is the unique lift of the path $h_t|_{p_0}$ starting at $q_0$. By Definition \ref{def:fractional dehn twist}, we have $c(b)=1+\frac{1}{4}=\frac{5}{4}$ and  $c(\psi)=\frac{5}{12}=\frac{1}{3}c(b)$.}
\label{fig:fractional Dehn twist lifting map}
\end{figure}
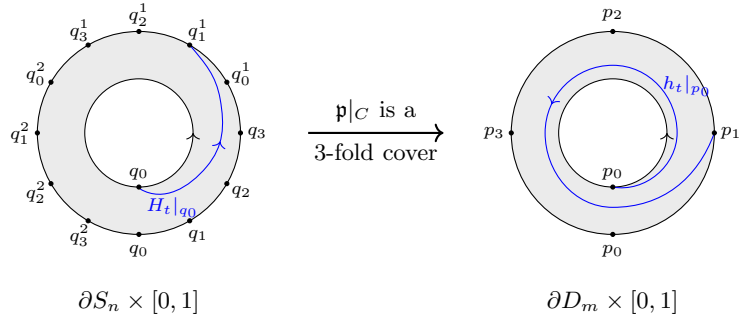

Let $c(b)=l+\frac{a}{N}$. By Definition \ref{def:fractional dehn twist}, this means that the path
$h_t|_{p_0}$ is homotopic to the path $[\partial D_m]^l\cdot\gamma_{p_0p_a}$ on $\partial D_m$. 

We write $l=ns+r$ where $0\leq r<n$. Then by the uniqueness of path lifting, $H_t|_{q_0}$ is homotopic to the path $C^s\cdot \gamma_{q_0q^r_0}\cdot \gamma_{q^r_0q^r_k}$, which is equal to $C^s\cdot\gamma_{q_0q^r_k}$. Therefore,
\begin{align*}
 c(\varphi) & =s+\frac{rN+a}{Nn} \\
 & = \frac{N(sn+r)+a}{Nn}=\frac{Nl+k}{Nn}\\
 & = \frac{1}{n}c(b). 
\end{align*}
This deals with the case that $(m, n) = 1$. 

In case that $(m,n)\neq 1$, the degree of the covering map $\mathfrak{p}: S_n\rightarrow D_m$ restricted to each boundary component of $S_n$ is $\frac{n}{(m,n)}$. Proceed as in the case that $(m,n) = 1$ to complete the proof. 
\end{proof}

\subsection{ The L-space conjecture and cyclic branched covers of closed braids}

In this section we study the left-orderability of branched covers of closed braids. We begin with the proof of Theorem \ref{thm:conjecture cyclic braids}: {\it Let $b\in B_{m}$ be an odd-strand braid whose closure $\hat{b}$ is an oriented hyperbolic link $L$ and let $c(b) \in \mathbb Q$ be the fractional Dehn twist coefficient of $b$. Suppose that $|c(b)| \geq 2$. Then all even order cyclic branched covers of $\hat{b}$ are excellent.}

\begin{proof}[Proof of Theorem \ref{thm:conjecture cyclic braids}] 

First of all, the equivariant sphere theorem (\cite{MSY}) and the positive solution of the Smith Conjecture (\cite{MB}) imply that as $\hat{b}$ is prime, all cyclic branched covers of $\hat{b}$ are irreducible. 

For each $n \geq 1$, there is an $n$-fold cyclic branched cover $\mathfrak{p}_n: \Sigma_{2n}(\hat{b}) \to \Sigma_{2}(\hat{b})$, branched over the lift $\tilde L$ of $\hat b$ to $\Sigma_{2n}(\hat{b})$, which the reader will verify is surjective on the level of fundamental groups. Hence if $b_1(\Sigma_2(\hat{b}))>0$, then $b_1(\Sigma_{2n}(\hat{b}))>0$ for all $n$. As such, Conjecture \ref{conj: lspace} holds for all even order cyclic branched covers of $\hat{b}$. 

Suppose then that $\Sigma_2(\hat{b})$ is a rational homology sphere. If $c(b)\leq -2$, then $c(b^{-1})=-c(b)\geq 2$ by Lemma \ref{lem:poincare translation number}. Since $\widehat{b^{-1}}$ is the mirror image of $\hat{b}$, their cyclic branched covers are diffeomorphic, so without loss of generality we may assume that $c(b)\geq 2$. 

By hypothesis, $b$ is a pseudo-Anosov mapping class of $D_{m}$. Then Lemma \ref{lem:lift pseudo Anosov map} shows that the $2$-fold cyclic branched cover of $\hat{b}$ admits an open book decomposition $(\bar{S}_2,\bar\psi)$, where $\bar{S}_2$ is the $2$-fold cyclic branched cover of the unit disk branched  over $m$ points and the monodromy $\bar\psi$ is pseudo-Anosov. By Lemma \ref{lem:FDTC of lifting monodromy},
\begin{displaymath}
 c(\bar{\psi})=\frac{c(b)}{2}\geq 1. 
\end{displaymath}
 Since $b$ is a braid on an odd number of strands, $\partial \bar{S}_2$ is connected.  Then by Theorem \ref{thm:cgeq1}, there exists a co-orientable taut foliation $\mathcal{F}$ on $\Sigma_2(\hat{b})$ and hence it cannot be an L-space \cite{OS1,Bn,KR2}. Moreover, Theorem \ref{thm:cgeq1} says that the tangent plane field of the foliation $\mathcal{F}$ is homotopic to the contact structure supported by the open book $(\bar{S}_2,\bar\psi)$. On the other hand, this contact structure is isotopic to the lift of the contact structure on $S^3$ that is supported by the open book $(\bar{D}_{m},\bar{b})$. (Here ``lift" is used in the sense of \S \ref{subsec: contact}.) Therefore, by Lemma \ref{lem:Euler class vanishes links} we have $e(T\mathcal{F})=0$. Applying Theorem \ref{prop:taut foliation left orderability}, we conclude that $\pi_1(\Sigma_2(\hat{b}))$ is left-orderable. This complete the proof for the $2$-fold cyclic branched cover.
 
Now consider $\Sigma_{2n}(\hat{b})$ where $n > 1$ and recall that there is an $n$-fold branched cyclic cover $\mathfrak{p}_n: \Sigma_{2n}(\hat{b}) \to \Sigma_{2}(\hat{b})$, branched over $\tilde L$. We noted above that $\mathfrak{p}_n$ is surjective on the level of $\pi_1$, so $\pi_1(\Sigma_{2k}(\hat{b}))$ is left-orderable by Theorem \ref{thm:brw}. Finally, since $\tilde{L}$ intersects the foliation $\mathcal{F}$ on $\Sigma_2(\hat b)$ transversely (cf. \cite[Lemma 4.4]{HKM2}), $\mathcal{F}$ lifts to a foliation $\mathcal{F}_n$ on $\Sigma_{2n}(\hat{b})$ which is easily seen to be co-oriented and taut. Consequently, $\Sigma_{2n}(\hat{b})$ is not an L-space for any $n > 1$, which completes the proof. 
\end{proof}

An identical argument yields the following more general statement. We omit the proof.  

\begin{theorem}
 Let $b\in B_m$ be a psudo-Anosov braid whose fractional Dehn twist coefficient satisfies $|c(b)|\geq N$. Then the $nk$-fold cyclic branched cover of the closed braid $\hat{b}$ admits a co-oriented taut foliation and has a left-orderable fundamental group for any $n$ with $2\leq n\leq N$, $(m,n)=1$ and $k\geq1$. 
 \qed
\label{thm:taut foliation in cyclic covers of a braid}
\end{theorem}

\subsection{Dehornoy's braid ordering and  cyclic branched covers of closed braids. }
\label{subsec:order of the braid group}

There is a special left order $<_D$ on the braid group $B_m$, due to Dehornoy, characterised by the condition that a braid $b$ is positive if and only if there is a $j \geq 1$ such that $b$ can be written in the standard braid generators (Figure \ref{fig:fig34}) as a word containing $\sigma_j$, but no $\sigma_j^{-1}$, and not containing $\sigma_i^{\pm 1}$ for $i < j$.  Set
$\Delta_m = (\sigma_1 \sigma_2 \ldots \sigma_{m-1})(\sigma_1 \sigma_2 \ldots \sigma_{m-2}) \ldots (\sigma_1 \sigma_2)(\sigma_1)$ in $B_m$. 
The centre of $B_m$ is generated by $\Delta_m^2$ and for each $b \in B_m$ there is an integer $d > 0$ such that 
$$\Delta_m^{-2d} <_D b <_D \Delta_m^{2d}$$
In other words, the subgroup of $B_m$ generated by $\Delta_m$ is cofinal in $B_m$ with respect to $<_D$. 

\begin{definition}
  Given an element $b\in B_m$, the Dehornoy floor $\lfloor b \rfloor_D$ is the nonnegative integer defined to be
  \begin{displaymath}
   \lfloor b\rfloor_D=\min\{k\in \mathbb{Z}_{\geq 0}\mid \Delta_m^{-2k-2}<_D b <_D \Delta_m^{2k+2}\}.
\end{displaymath}   
\label{def:Dehornoy floor}
 \end{definition}
\vspace{-.2cm} 

Although the Dehornoy's floor fails to be a topological invariant of the closed braid, it has proven to be a useful concept when studying links via closed braids \cite{Ito1,Ito2}. 

Malyutin discovered a fundamental relationship between $c(b)$ and $\lfloor b\rfloor_D$. Though the Dehornoy's floor defined by Malyutin \cite[Definition 7.3]{Mal} is slightly different from the one given above, it is easy to check that these two agree for $b>_D 1$ in $B_m$. 

\begin{proposition}{\rm (cf. \cite[Lemma 7.4]{Mal})} \label{prop: malyutin} 
For each $b \in B_m$, $\lfloor b\rfloor_D \leq |c(b)| \leq \lfloor b\rfloor_D + 1$. 
\label{prop:Dehornoy floor FDTC}
\end{proposition}

 \begin{proof}
 Let $d= \lfloor b\rfloor_D$. 
 
 By Definition \ref{def:Dehornoy floor}, this means either $\Delta_m^{2d} \leq_D b <_D \Delta_m^{2d+2}$ or $\Delta_m^{-2d-2 } <_D b \leq_D \Delta_m^{-2d}$. 
 In the first case the claim follows directly from \cite[Lemma 7.4]{Mal}. In the second, the fact that $\Delta_m^2$ lies in the center of $B_m$ implies that $\Delta_m^{2d} \leq_D b^{-1} <_D \Delta_m^{2d + 2}$, so by the first case,  $\lfloor b\rfloor_D \leq_D c(b^{-1})\leq_D \lfloor b \rfloor_D+1$. Since $c(b^{-1})=-c(b)=|c(b)|$, this completes the proof.
 \end{proof}
 
 \begin{remark}\label{rmk:FDTC braid ordering}
  The proof of Proposition \ref{prop:Dehornoy floor FDTC} shows that $c(b) \geq 0$ when $b>_D 1$ and $c(b) \leq 0$ when $b<_D 1$. 
 \end{remark}

Given Proposition \ref{prop:Dehornoy floor FDTC}, Theorem \ref{thm:conjecture cyclic braids} has the following consequence. 

\begin{theorem}
 Let $b\in B_{m}$ be an odd-strand pseudo-Anosov braid. Suppose that $\lfloor b \rfloor_D>1$. Then all even order cyclic branched covers of $\hat{b}$ admit co-oriented taut foliations and have left-orderable fundamental groups. 
 \qed
 \label{thm:conjecture cyclic braids braid ordering}
\end{theorem}

%%%%%%%%%%%%%%%%%%%%%%%%%%%%%%%%%%%%%%%%%%%
%%%% We are considering B_3 after this %%%%
%%%%%%%%%%%%%%%%%%%%%%%%%%%%%%%%%%%%%%%%%%%

\section{The L-space conjecture and genus one open books}
\label{sec:L-space conjecture genus one open books}
The goal of this section is to prove Theorem \ref{thm:lspace genus one open book decomposition} and Corollary \ref{cor: branched cover genus 1}.

A simple Euler characteristic calculation shows that the $2$-fold branched cover of a disk branched over three points is a genus $1$ surface with one boundary component (see Figure \ref{fig:double cover 3-punctured disk}). We denote this surface by $T_1$ and let $\theta$ be its covering involution.

\begin{figure}
\centering
 \begin{tikzpicture}
 % \draw [help lines] (0,0) grid (14,5);
   %% right circle 
  % \filldraw [ultra-light-gray] (10.5,2.5) circle (1.1);  
   \draw [thick,->] (11.6,2.5) [out=90,in=0] to (10.5,3.6);
   \draw [thick] (10.5,3.6) [out=180, in=90] to (9.4,2.5) [out=270,in=180] to (10.5,1.4) [out=0,in=270] to (11.6, 2.5);
   \draw [red,thick] (9.83,2.5) -- (10.47,2.5);
   \draw [blue,thick] (11.17,2.5) -- (10.53,2.5);
   \filldraw (10.5,2.5) circle (0.03);
   \filldraw  (9.8,2.5) circle (0.03);
   \filldraw (11.2,2.5) circle (0.03);
   \draw (10.5,2.5) circle (0.03);
   \draw (9.8,2.5) circle (0.03);
   \draw (11.2,2.5) circle (0.03);
   %% the middle arrow 
   \draw [->, thick] (6.6,2.5) -- (8.2, 2.5);
   %% rotation axis
   \draw [dashed, gray] (0.5,2.5) -- (6,2.5);
%% fill the lower half of the torus
  %  \filldraw [ultra-light-gray] (1.5,1.6) [out=0,in=260] to (1.9,2.1) [out=0, in=280] to (2.7,2.5) [out=-60, in=240] to (3.8,2.5) [out=-45, in=225] to (5,2.5) [out=270, in=180] to (3.4, 1.2) [out=180, in=-10] to (1.5,1.6);
   %% left torus 
   \draw [thick] (1.5,3.4) [out=10,in=180] to (3.4,3.8) [out=0,in=90] to (5,2.5); 
   \draw [thick] (5,2.5) [out=270, in=0] to (3.4, 1.2) [out=180, in=-10] to (1.5,1.6);   
   %% blue c_2
   \draw [thick, blue] (3.8,2.5) [out=-45, in=225] to (5,2.5);
   \draw [thick, dashed, blue] (5,2.5) [out=85, in=105] to (3.8,2.5);
   \node [below] at (4.5,2.2) {\tiny \color{blue} $c_2$};
   %% middle circle c_1
    \node [below] at (3.25, 2.2) {{\tiny \color{red}$c_1$}};
   \draw [thick, red]  (2.7,2.5) [out=-60, in=240] to (3.8,2.5);
   \draw [thick, red] (2.7,2.5) [out=60, in=120] to (3.8,2.5);
   %% the left curve 
   \draw [gray,thick](1.9,2.1) [out=0, in=280] to (2.7,2.5);
   \draw [dashed,gray, thick] (2.7,2.5) [out=115, in=0] to (1.1,2.95);
   %% leftmost circle
   \draw [thick] (1.5,1.6) [out=0,in=260] to (1.9,2.1) [out=80, in=-80] to (1.9, 2.9) [out=100, in=0] to (1.5, 3.4);
   \draw [thick, ->] (1.94,2.5) [out=270, in=90] to (1.94,2.55);
   \draw [thick] (1.5, 3.4) [out=195, in=80] to (1.1,2.9) [out=260, in=100] to (1.1, 2.1) [out=-80, in=160] to (1.5,1.6);
   %% three punctures 
   \filldraw (3.8,2.5) circle (0.03);
   \filldraw  (2.7,2.5) circle (0.03);
   \filldraw  (5,2.5) circle (0.03);
   \draw (3.8,2.5) circle (0.03);
   \draw (2.7,2.5) circle (0.03);
   \draw (5,2.5) circle (0.03);
   %% involution arrow 
   \draw [->] (5.45,2.35) [out=-60,in=180] to (5.6,2.2) [out=0, in=270] to (5.75,2.5) [out=90, in=0] to (5.6, 2.8) [out=180,in=60] to (5.45,2.65);
   \draw [white] (0,4.5) circle (0.1);
   \draw [white] (0,0.9) circle (0.1);
    \draw [white] (12.7,2.5) circle (0.1);
    \node at (5.9,2.9) {$\theta$};
 \end{tikzpicture}
 \caption{$2$-fold cover of $3$-punctured disk}
\label{fig:double cover 3-punctured disk}
 \end{figure}
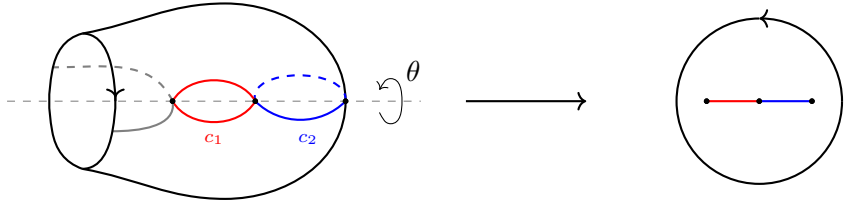

As in \S \ref{subsec:lift open book}, every element in $B_3 = \hbox{Mod}(D_3)$ admits a unique lift to ${\rm Mod}(T_1)$ which defines an embedding of groups $B_3 \to \hbox{Mod}(T_1)$. The Artin generators $\sigma_1$ and $\sigma_2$ of $B_3$ lift to the right-handed Dehn twists $T_{c_1}$ and $T_{c_2}$ respectively, where $c_i$ is the preimage of the segment connecting the $i^{th}$ and the $(i+1)^{st}$ punctures of the disk for $i=1,2$, as in Figure \ref{fig:double cover 3-punctured disk}. Since the Dehn twists $T_{c_1}$ and $T_{c_2}$ generate ${\rm Mod}(T_1)$, the embedding $B_3 \to \hbox{Mod}(T_1)$ constructed above is an isomorphism.

The identification between $B_3$ and the mapping class group ${\rm Mod}(T_1)$ can be used to show that any $3$-manifold which admits an open book decomposition with $T_1$-pages is the $2$-fold cyclic branched cover of a closed $3$-braid. Hence, the following conjugacy classification of $3$-braids leads to a complete list of diffeomorphism classes of genus $1$ open books with connected binding. 

Set $C=\Delta_3^2=(\sigma_1\sigma_2)^3$. 

\begin{theorem} {\rm (\cite[Proposition 2.1]{Mur}, \cite[Theorem 2.2]{Bal})} 
\label{thm:classification of 3 braids}
Up to conjugation, any braid $b$ in $B_3$ is equal to one of the following: 
\vspace{-.2cm} 
\begin{enumerate}
\item $C^d\cdot\sigma_1\sigma_2^{-a_1}\cdots\sigma_1\sigma_2^{-a_n}$, where $d\in \mathbb{Z}$, $a_i\geq 0$ and at least one of the $a_i$ is nonzero; 
\vspace{.2cm} \item $C^d\cdot \sigma_2^m$ for some $d\in \mathbb{Z}$ and $m\in \mathbb{Z}$. 
\vspace{.2cm} \item $C^d\cdot \sigma_1^m\sigma_2^{-1}$, with $d\in \mathbb{Z}$ and $m\in \{-1,-2,-3\}$. \qed 
\end{enumerate}
\end{theorem}

\begin{remark}
The conjugation classification of $3$-braids detailed in Theorem \ref{thm:classification of 3 braids} corresponds to Nielsen-Thurston classification of mapping classes in ${\rm Mod}(D_3)=B_3$ (\S \ref{subsec:nielsen-thurson classification}): braids of type (1) are pseudo-Anosov; those of type (2) are reducible; those of type (3) are periodic. This can be verified by considering their lifts in ${\rm Mod}(T_1)$, where the Nielsen-Thurston type of an element $h\in {\rm Mod}(T_1)$ is determined by the trace of the linear map $h_*: H_1(T_1)\rightarrow H_1(T_1)$ (cf. \cite[\S 13.1]{FM})\footnote{The Nielsen-Thurston type of an element $h\in {\rm Mod}(T_1)$ is, by definition, the same as that of its projection in the mapping class group of the once-punctured torus, which is isomorphic to $SL(2,\mathbb{Z})$.}. 
\end{remark}

Baldwin listed all closed $3$-braids whose $2$-fold branched covers are L-spaces. 
 
\begin{theorem}{\rm (Theorem 4.1 of \cite{Bal})} \label{thm: baldwin}
The $2$-fold cyclic branched cover $\Sigma_2(\hat{b})$ of a closed braid $\hat{b}$ for $b\in B_3$ is an L-space if and only if $b$ belongs to one of the following lists: 
\vspace{-.2cm}
\begin{enumerate}
\item $C^d\cdot\sigma_1\sigma_2^{-a_1}\cdots\sigma_1\sigma_2^{-a_n}$, where $a_i\geq 0$ and at least one of the $a_i$ is nonzero, and $d\in\{-1,0,1\}$. 
\vspace{.2cm} \item $C^d\cdot \sigma_2^m$,  for some $m\in \mathbb{Z}$ and $d=\pm 1$. 
\vspace{.2cm} \item $C^d\cdot \sigma_1^m\sigma_2^{-1}$, with $m\in \{-1,-2,-3\}$ and $d\in\{-1,0,1,2\}$. 
\end{enumerate}
\end{theorem}
Conjecture \ref{conj: lspace} holds for the manifolds listed in Theorem \ref{thm: baldwin} since they admit no co-oriented taut foliations (\cite{OS1,Bn,KR2}) and it was shown by Li and Watson (\cite{LW}) that they have non-left-orderable fundamental groups.   

\begin{proof}[Proof of Theorem \ref{thm:lspace genus one open book decomposition}] 
Let $M$ denote an irreducible $3$-manifold which admits a genus $1$ open book decomposition with connected binding. Then $M$ is diffeomorphic to the $2$-fold cyclic branched cover of a closed $3$-braid $\hat{b}$, where $b\in B_3$ falls into one of the three families listed in Theorem \ref{thm:classification of 3 braids}. 

Suppose first that $b$ is in family (3). Then $b=C^d\sigma_1^m\sigma_2^{-1}$ with $m\in\{-1,-2,-3\}$. We noted above that this implies that $b$ is periodic, so the mapping torus $N_b \cong D_m\times [0,1]/(x,1)\sim (b(x),0)$ of $b$ is Seifert fibred. The reader will verify that $b$ has period $3, 4$, or $6$ depending on whether $m$ is $-3, -2$, or $-1$. This implies that the fibre class on $\partial N_b$ intersects $\nu$ (Figure \ref{fig:branched_cover_braid_fig0}) more than once algebraically. In particular, $\nu$ is not the fibre class. Thus $N_b(\nu)$, the exterior of $\hat b$ in $M$, is a Seifert manifold and therefore so is the irreducible manifold $\Sigma_2(\hat b)$. (The reader will verify that $\hat b$ is prime and so the equivariant sphere theorem implies that $\Sigma_2(\hat b)$ is irreducible.) Hence the theorem holds in this case by \cite{BRW, BGW, LS}. 

If $b$ is in family (2), then $b = C^d \cdot \sigma_2^m$ with $m\in \mathbb{Z}$. Note that $d \ne 0$ as otherwise $\hat b$ is a split link and therefore $\Sigma_2(\hat b)$ is reducible, contrary to our hypotheses. Hence, given the paragraph immediately after Theorem \ref{thm: baldwin}, we can assume that $|d| \geq 2$. We can also suppose that $m \ne 0$ as otherwise $\hat b$ is a $(3, 3d)$ torus link so that $\Sigma_2(\hat b)$ is Seifert fibred and the theorem's conclusion follows from \cite{BRW, BGW, LS}. 

There is a circle $O$ in $\hbox{int}(D_3)$ which contains the second and third punctures of $D_3$ in its interior, but not the first, and is invariant under $b$. If $T$ denotes the torus obtained from $O$ in the mapping torus of $b$, the reader will verify that there is a genus $1$ Heegaard splitting $V_1 \cup_T V_2$ of $S^3$ where $V_1 \cap \hat b$ is an $(2, m+2d)$ torus link standardly embedded in the interior of $V_1$ and $V_2 \cap \hat b$ is a $(d,1)$ torus knot standardly embedded in the interior of $V_2$.

Then $\Sigma_2(\hat b)$ is the union of a $2$-fold branched cover $\Sigma_2(V_1, V_1 \cap \hat b)$ of $V_1$ branched over $V_1 \cap \hat b$ and a $2$-fold branched cover $\Sigma_2(V_2, V_2 \cap \hat b)$ of $V_2$ branched over $V_2 \cap \hat b$. Since $V_1 \setminus \hat b$ fibres over the circle with fibre a twice-punctured disc, $\Sigma_2(V_1, V_1 \cap \hat b)$ is homeomorphic to the product of a torus and an interval. On the other hand, $V_2$ admits a Seifert structure for which $V_2 \cap \hat b$ is a regular Seifert fibre, and therefore $\Sigma_2(V_2, V_2 \cap \hat b)$ admits a Seifert structure. Further, as $|d| \geq 2$, $\Sigma_2(V_2, V_2 \cap \hat b)$ is not a solid torus and so has an incompressible boundary. Thus $\Sigma_2(\hat b)$ is a graph manifold. In this case, Conjecture \ref{conj: lspace} has been confirmed in \cite{BC,HRRW}. 

Finally suppose that $b$ is in family (1). Then $b = C^d\cdot \sigma_1 \sigma_2^{-a_1} \cdots \sigma_1 \sigma_2^{-a_n}$ is pseudo-Anosov, where $a_i\geq 0$ and $a_i\neq 0$ for some $i$. As some of the $a_i$ may be zero, we can write $b = C^d \cdot \sigma_1^{b_1} \sigma_2^{-c_1} \cdots \sigma_1^{b_k} \sigma_2^{-c_k}$ where each $b_i$ and each $c_i$ is positive. Further, $k \geq 1$. We claim that the fractional Dehn twist coefficient $c(\sigma_1^{b_1} \sigma_2^{-c_1} \cdots \sigma_1^{b_k} \sigma_2^{-c_k})$ is zero. 

To see this, first observe that the conjugation by $\Delta_3=\sigma_1 \sigma_2 \sigma_1$ in $B_3$ exchanges $\sigma_1$ and $\sigma_2$.  Then by Lemma \ref{lem:poincare translation number}
\begin{equation}
 c(\sigma_2^{b_1} \sigma_1^{-c_1} \cdots \sigma_2^{b_k} \sigma_1^{-c_k}) = c(\Delta_3 (\sigma_1^{b_1} \sigma_2^{-c_1} \cdots \sigma_1^{b_k} \sigma_2^{-c_k}) \Delta_3^{-1} ) = c(\sigma_1^{b_1} \sigma_2^{-c_1} \cdots \sigma_1^{b_k} \sigma_2^{-c_k}).
 \label{equ:FDTC zero}
\end{equation}

But by the definition of the Dehornoy's order,  $\sigma_2^{b_1} \sigma_1^{-c_1} \cdots \sigma_2^{b_k} \sigma_1^{-c_k}<_D 1 <_D \sigma_1^{b_1} \sigma_2^{-c_1} \cdots \sigma_1^{b_k} \sigma_2^{-c_k}$. Hence according to Remark \ref{rmk:FDTC braid ordering}, Equality (\ref{equ:FDTC zero}) shows
\begin{displaymath}
 0 \leq c(\sigma_1^{b_1} \sigma_2^{-c_1} \cdots \sigma_1^{b_k} \sigma_2^{-c_k}) =  c(\sigma_2^{b_1} \sigma_1^{-c_1} \cdots \sigma_2^{b_k} \sigma_1^{-c_k}) \leq 0.
\end{displaymath}
Thus $c(\sigma_1^{b_1} \sigma_2^{-c_1} \cdots \sigma_1^{b_k} \sigma_2^{-c_k}) = 0$. But then as $C^d$ commutes with $\sigma_1^{b_1} \sigma_2^{-c_1} \cdots \sigma_1^{b_k} \sigma_2^{-c_k}$, 
$$c(b) = c(C^d) +c(\sigma_1\sigma_2^{-a_1}\cdots \sigma_1\sigma_2^{-a_m}) = d+c(\sigma_1\sigma_2^{-a_1}\cdots \sigma_1\sigma_2^{-a_m}) = d$$
by Lemma \ref{lem:poincare translation number}. 

By the discussion immediately preceding the statement of the theorem we can suppose that $|c(b)| = |d| \geq 2$. Theorem \ref{thm:conjecture cyclic braids} then implies that $\Sigma_2(\hat{b})$ is not an L-space, admits a co-orientable taut foliation, and has left-orderable  fundamental group. 
\end{proof}

Let $b(h)\in B_3$ be the image of $h$ under the isomorphism from ${\rm Mod}(T_1)$ to $B_3$ described above. Note that $b(\delta)=C$, where $\delta=(T_{c_1}T_{c_2})^3$ as in Corollary \ref{cor: branched cover genus 1}. 

\begin{proof}[Proof of Corollary \ref{cor: branched cover genus 1}] 
Suppose that $K$ is a genus one fibred knot in a closed, connected, orientable, irreducible $3$-manifold $M$, with fibre $T_1$ and monodromy $h$. Since $K$ has genus one, it is prime and therefore each $\Sigma_n(K)$ is irreducible. Further, each $\Sigma_n(K)$ contains a genus one fibred knot so Theorem \ref{thm:lspace genus one open book decomposition} implies that $\Sigma_n(K)$ is excellent, respectively a total L-space, for some $n \geq 2$ if and only if it is not an L-space, respectively is an L-space. 

 Up to conjugation, we can suppose that the braid $b(h) \in B_3$ is one of the forms listed in Theorem \ref{thm:classification of 3 braids}.  

First suppose that $h$ is pseudo-Anosov. If $c(h) \ne 0$, Theorem \ref{thm:conjecture fibre knots}(1) implies that $\Sigma_n(K)$ is excellent for all $n \geq 2$. Suppose then that $c(h) = 0$ and let $n \geq 2$. Since $h$ is pseudo-Anosov, $b(h)=C^d\cdot\sigma_1\sigma_2^{-a_1}\cdots\sigma_1\sigma_2^{-a_n}$, where $d \in \mathbb{Z}$, $a_i\geq 0$ and at least one of the $a_i$ is nonzero. The proof of Theorem \ref{thm:lspace genus one open book decomposition} and Lemma \ref{lem:FDTC of lifting monodromy} show that 
\begin{equation} \label{eqn: c(f) via c(b(f))}
0 = 2c(h) = c(b(h)) = d.
\end{equation}
and therefore, $b(h^n)=b(h)^n=(\sigma_1\sigma_2^{-a_1}\cdots\sigma_1\sigma_2^{-a_n})^n$. Theorem \ref{thm: baldwin}(1) now implies that $\Sigma_n(K)$ is an L-space. This is case (1) of Corollary \ref{cor: branched cover genus 1}.   

Next suppose that $h$ is reducible. Then $b(h) = C^d\cdot \sigma_2^m$ for some $d \in \mathbb{Z}$ and $m\in \mathbb{Z}$ (Theorem \ref{thm:classification of 3 braids}(2)). It follows that $b(h^n)=b(h)^n=C^{nd}\cdot \sigma_2^{nm}$ and since $nd \ne \pm 1$ for $n \geq 2$, Theorem \ref{thm: baldwin} shows that $\Sigma_n(K)$ is not an L-space. Thus it is excellent. 

Finally suppose that $h$ is periodic. Then $b(h) = C^d\cdot \sigma_1^m\sigma_2^{-1}$ where $d \in \mathbb{Z}$ and $m\in \{-1,-2,-3\}$ (Theorem \ref{thm:classification of 3 braids}(3)). The reader will verify that if $w_m = \sigma_1^{-m} \sigma_2^{-1}$, then 
$$w_m^r = \left\{ \begin{array}{ll} C^{-1} & \hbox{ if $m = 1$ and $r = 3$} \\ 
 C^{-1} & \hbox{ if $m = 2$ and $r = 2$} \\  
C^{-2} & \hbox{ if $m = 3$ and $r = 3$} 
\end{array} \right.$$
In particular,
$$w_m^2 = \left\{ \begin{array}{ll} C^{-1}w_{1}^{-1} & \hbox{ if $m = 1$} \\ 
 C^{-2}w_{3}^{-1} & \hbox{ if $m = 3$} 
\end{array} \right.$$
We consider the cases $m = 1, 2, 3$ separately. 

Suppose that $m = 1$ and let $n = 3k+r \geq 2$ where $r \in \{0,1,2\}$. Then $k \geq 1$ if $r$ is $0$ or $1$ and $k \geq 0$ otherwise. The identities above imply that 
$$b^n = (C^{d}w_{1})^{n}  = \left\{ \begin{array}{ll} C^{nd - k} & \hbox{ if } r = 0 \\ 
 C^{nd - k} w_{1} & \hbox{ if } r = 1 \\  
 (C^{k + 1 - nd} w_{1})^{-1}& \hbox{ if } r = 2 
\end{array} \right.$$
Theorem \ref{thm:classification of 3 braids} then shows  
$$\hbox{\rm $\Sigma_n(K)$ is an L-space if and only if } \left\{ \begin{array}{cl} 
nd  = k \pm 1 & \hbox{ {\rm if }} r = 0 \\

k-1 \leq nd  \leq k+2 & \hbox{ {\rm if }} r = 1, 2 

\end{array} \right. $$
Thus, 
$$\hbox{\rm $\Sigma_n(K)$ is an L-space if and only if } b =  \left\{ \begin{array}{r} 
w_{1}  \hbox{ {\rm and }} n \leq 5 \\ 

Cw_{1}  \hbox{ {\rm and }} n = 2 

\end{array} \right. $$

Next suppose that $m = 2$ and let $n = 2k+r \geq 2$ where $r \in \{0,1\}$. Then $k \geq 1$ and 
$$b^n = (C^{d}w_{2})^{n}  = \left\{ \begin{array}{ll} C^{nd - k} & \hbox{ if } r = 0 \\ 
C^{nd - k} w_{2} & \hbox{ if } r = 1 
\end{array} \right.$$
Theorem \ref{thm:classification of 3 braids} then shows  
$$\hbox{\rm $\Sigma_n(K)$ is an L-space if and only if } \left\{ \begin{array}{cl} 
nd  = k \pm 1 & \hbox{ {\rm if }} r = 0 \\

k-1 \leq nd  \leq k+2 & \hbox{ {\rm if }} r = 1

\end{array} \right. $$
Thus, 
$$\hbox{\rm $\Sigma_n(K)$ is an L-space if and only if } b = \left\{ \begin{array}{r} 
w_{2}  \hbox{ {\rm and }} n \leq 3 \\ 

Cw_{2}  \hbox{ {\rm and }} n \leq 3  

\end{array} \right. $$

Finally suppose that $m = 3$ and let $n = 3k+r \geq 2$ where $r \in \{0,1,2\}$. Then $k \geq 1$ if $r$ is $0$ or $1$ and $k \geq 0$ otherwise. We have, 
$$b^n = (C^{d}w_{3})^{n}  = \left\{ \begin{array}{ll} C^{nd - 2k} & \hbox{ if } r = 0 \\ 
C^{nd - 2k} w_{3} & \hbox{ if } r = 1 \\  
 (C^{2(k+1) - nd} w_{3})^{-1} & \hbox{ if } r = 2 
\end{array} \right.$$
Theorem \ref{thm:classification of 3 braids} then shows  
$$\hbox{\rm $\Sigma_n(K)$ is an L-space if and only if } \left\{ \begin{array}{cl} 
nd  = 2k \pm 1 & \hbox{ {\rm if }} r = 0 \\

2k-1 \leq nd  \leq 2k+2 & \hbox{ {\rm if }} r = 1 \\  

2k \leq nd  \leq 2k+3 & \hbox{ {\rm if }} r = 2 

\end{array} \right. $$
Thus, 
$$\hbox{\rm $\Sigma_n(K)$ is an L-space if and only if } b = \left\{ \begin{array}{r} 
w_{3}  \hbox{ {\rm and }} n = 2\\ 

Cw_{3}  \hbox{ {\rm and }} n \leq 5 

\end{array} \right. $$
This completes the proof. 
\end{proof}

\section{The left-orderability of the fundamental groups of cyclic branched covers of satellite links}
\label{sec:LO cyclic cover satellite knots}
Let $L$ be an oriented untwisted satellite link in an oriented integer homology $3$-sphere $M$ with pattern $P$, a link in the solid torus $N$, and companion $C$, a knot in $M$. The $n$-fold cyclic branched cover of $L$ is obtained in the usual way from the regular cover of the exterior of $L$ determined by the homomorphism which sends the oriented meridians of the components of $L$ to $1$ (mod $n$). It can be obtained by gluing copies of the cyclic covers of the knot exterior $X(C)$ to an $n$-fold cyclic cover of the solid torus $N$ branched over $P$. 

In what follows we assume that $L, P$ and $C$ are as above and that $P$ is the closure of an $m$-strand braid pseudo-Anosov $b$. The reader will verify that $H_1(N \setminus P) \cong \mathbb Z^{|P| + 1}$, where $|P|$ is the number of components of $P$, is freely generated by the meridian classes of the components of $P$ and the class $\nu$ carried by a longitudinal loop on $\partial N = \partial X(C)$ (see Figure \ref{fig:branched_cover_braid_fig0}). We use $N_n(P)$ to denote the $n$-fold cyclic branched cover of $P$ in $N$ determined by the homomorphism $H_1(N \setminus P) \to  \mathbb Z/n$ which sends the meridians of the components of $P$ to $1$ (mod $n$) and $\nu$ to $0$ (mod $n$). 

We follow the notation developed in Section \ref{subsec:lift open book}. Our discussion there shows that $N_n(P)$ is homeomorphic to the mapping torus $\bar{S}_n\times [0,1]/(x,1)\sim (\bar\psi(x),0)$, where $\bar{S}_n$ is the $n$-fold cyclic branched cover of the disk branched at $m$ points and $\bar{\psi}$ is the unique lift of $b$ satisfying $\bar{\psi}|_{\partial \bar {S}_n}$ is the identity. (See \S \ref{subsec:lift open book} for the details.) The boundary of $\bar{S}_n$ is connected when $\gcd(m,n)=1$, so $\partial N_n(P)$ is a torus. It is clear that each of the curves $\partial \bar{S}_n\times t_0$ carries the longitudinal slope of $N_n(P)$. Let $\mu$ be the class in $H_1(\partial N_n(P))$ carried by $(p_0 \times [0,1])/\sim$, $p_0\in \partial \bar{S}_n$.

\begin{proof}[Proof of Theorem \ref{thm: satellite knot c(b) and c(h) nonnegative}]
By assumption, $(S, h^n)$ is an open book decomposition of $X_n(C)$. As in \S \ref{sec: fdtc and lo} we use $\mu_n$ and $\lambda_n$ to denote the meridional and longitudinal slopes on $\partial X_n(C)$. Since $\gcd(m,n)=1$, we have 
$$\Sigma_n(L)=N_n(P)\cup_{\varphi} X_n(C)$$ 
where the homeomorphism $\varphi: \partial N_n(P)\rightarrow \partial X_n(C)$ sends $\mu$ to $\lambda_n$ and $\lambda$ to $\mu_n$. We will show that $\Sigma_n(L)$ admits a co-orientable taut foliation and has a left-orderable fundamental group (and is therefore excellent) by showing that the Dehn fillings $N_n(P)(\mu-\lambda)$ and $X_n(C)(\mu_n-\lambda_n)$ have the same properties.
 
Consider $X_n(C)$ first. The Dehn filling $X_n(C)(\mu_n-\lambda_n)$ has an open book decomposition $(S, T_\partial \circ h)$, where $T_\partial$ denotes the right-handed Dehn twist along $\partial S$. The assumption that $c(h)\geq 0$ implies  $nc(h)+1\geq 1$. By Theorem \ref{thm:conjecture fibre knots}(1) we know that $X_n(C)(\mu_n-\lambda_n)$ is excellent and hence has left-orderable fundamental group. However, to obtain a foliation on $\Sigma_n(L)$, we need to use that fact that Theorem \ref{thm:cgeq1} guarantees a co-oriented taut foliation on $X_n(C)(\mu_n-\lambda_n)$ which is transverse to the binding of the open book $(S, T_\partial \circ h)$. In our case, the binding is the core of the filling torus and consequently there is a co-oriented taut foliation on $X_n(C)$ which restricts to a linear foliation on $\partial X_n(C)$ of slope $\mu_n-\lambda_n$. 
 
Next consider $N_n(P) \cong \bar{S}_n\times [0,1]/(x,1)\sim (\bar\psi(x),0)$. (See the discussion just prior to the proof of Theorem \ref{thm: satellite knot c(b) and c(h) nonnegative}.) By Lemma \ref{lem:FDTC of lifting monodromy}, the assumption $c(b)\geq 0$ implies that $c(\bar{\psi})=\frac{c(b)}{n}\geq 0$. Since $N_n(P)(\mu-\lambda)$ is homeomorphic to the open book $(\bar{S}_n,T_\partial\circ\bar{\psi})$, an argument analogous to that used in the previous paragraph shows that there is a co-orientable taut foliation $\mathcal{F}_0$ on $N_n(P)$ which intersects $\partial N_n(P)$ transversely in a linear foliation of slope $\mu-\lambda$. 

We need to show that $\pi_1(N_n (P)(\mu-\lambda))$ is left-orderable. Let $\mathcal{F}$ denote the co-orientable taut foliation on $N_n(P)(\mu-\lambda)$ induced by $\mathcal{F}_0$. The existence of $\mathcal{F}$ implies that $N_n(P)(\mu-\lambda)$ is irreducible (\cite{Nov}). Hence, if $b_1(N_n(P)(\mu-\lambda)) > 0$, this is true by Corollary \ref{cor:b1 is not zero lo}. Assume then that $b_1(N_n(P)(\mu-\lambda)) = 0$. In other words, $N_n(P)(\mu-\lambda)$ is a rational homology $3$-sphere. We show that $e(T\mathcal{F})=0$.

 It is clear from our argument that the tangent plane field $T\mathcal{F}$ is homotopic to the contact structure supported by the open book $(\bar{S}_n,T_\partial\circ\bar{\psi})$ (cf. Theorem \ref{thm:cgeq1}). We observe that the monodromy $T_\partial\circ\bar{\psi}: \bar{S}_n\rightarrow \bar{S}_n$ is the lift of the braid  $\Delta_m^{2n}b:D_m\rightarrow D_m$ (see \S \ref{subsec:lift open book}). Hence by Lemma \ref{lem:Euler class vanishes links}, the contact structure supported by $(\bar{S}_n,T_\partial\circ\bar{\psi})$ has zero Euler class and therefore so does $e(T\mathcal{F})$. As a result, $\pi_1(N_n(P)(\mu-\lambda))$ is left-orderable by Theorem \ref{prop:taut foliation left orderability}. 

Finally, by piecing together the foliations on $N_n(P)$ and $X_n(C)$ constructed above we obtain a co-orientable taut foliation on $\Sigma_n(L)$. Further, an application of \cite[Theorem 2.7]{CLW} implies that $\pi_1(\Sigma(L))$ is left-orderable. Thus $\Sigma_n(L)$ is excellent. 
\end{proof}

\begin{proof}[Proof of Theorem \ref{thm:satellite c(h)>0 n>>0}] 
The proof of this theorem is similar to the one used to prove Theorem \ref{thm: satellite knot c(b) and c(h) nonnegative}. As such, we only point out a few of the key points.

Note that when $c(h)\neq 0$ and $n\geq \frac{2}{|c(h)|}$, the fractional Dehn twist coefficient of the monodromy of the $n$-fold cyclic cover $X_n(C)$ satisfies $|c(h^n)|=n|c(h)|\geq 2$. Then the argument used in the proof of Theorem \ref{thm: satellite knot c(b) and c(h) nonnegative} to analyse $X_n(C)(\mu_n - \lambda_n)$ can be used to show that:
\vspace{-.2cm} 
\begin{itemize}

\item $\pi_1(X_n(C)(\mu_n+ \lambda_n))$ and $\pi_1(X_n(C)(\mu_n- \lambda_n))$ are left-orderable ; 

\vspace{.2cm} \item there is a co-orientable taut foliation on $X_n(C)$ which induces a linear foliation on $\partial X_n(C)$ with leaves of slope $\mu_n + \lambda_n$; 

\vspace{.2cm} \item there is a co-orientable taut foliation on $X_n(C)$ which induces a linear foliation on $\partial X_n(C)$ with leaves of slope $\mu_n - \lambda_n$.

\end{itemize}
\vspace{-.2cm} When $c(h) = 0$, it is simple to see that $X_n(C)$ possess the same properties.  

Based on these observations, the rest of the proof follows exactly as in that of Theorem \ref{thm: satellite knot c(b) and c(h) nonnegative} when $c(b)\geq 0$. In the case that $c(b)< 0$, one can apply the arguments of Theorem \ref{thm: satellite knot c(b) and c(h) nonnegative} but replacing $\mu_n - \lambda_n$ by $\mu_n+\lambda_n$ and $\mu - \lambda$ by $\mu+\lambda$ to complete the proof. 
\end{proof}


\begin{thebibliography}{HRRWS} 

{\small

\bibitem[Al]{Al} J.~W.~Alexander, {\it A Lemma on System of Knotted Curves}, Proc. Nat. Acad. Sci. USA {\bf 9} (1923), 93--95.

\bibitem[BM]{BM} K.~Baker and K.~Motegi, {\it Seifert vs slice genera of knots in twist families and a characterization of braid axes}, preprint (2017), arXiv:1705.10373.

\bibitem[Bal]{Bal} J.~Baldwin, {\it Heegaard Floer homology and genus one, one-boundary component open books}, J. Topology {\bf 4} (2008), 963--992.

\bibitem[BBG]{BBG}  M.~Boileau, S.~Boyer and C.~McA.~Gordon, {\it Branched covers of quasipositive links and L-spaces}, preprint (2017), arXiv:1710.07658. 

\bibitem[BC]{BC} S.~Boyer and A.~Clay, {\it Foliations, orders, representations, L-spaces and graph manifolds}, Adv. Math., {\bf 310} (2017), 159--234.

\bibitem[BGW]{BGW}  S.~Boyer, C.~McA.~Gordon and L.~Watson, {\it  On L-spaces and left-orderable fundamental groups}, Math. Ann. {\bf 356} (2013), 1213--1245. 

\bibitem[BPH]{BPH} M.~Boileau, J.~Porti and M.~Heusener, {\bf Geometrization of 3-orbifolds of cyclic type}, Ast\'erisque {\bf 272}, Soc. Math. France, Paris, 2001.

\bibitem[BRW]{BRW}  S.~Boyer, D.~Rolfsen and B.~Wiest, {\it  Orderable 3-manifold groups}, Ann. Inst. Fourier {\bf 55} (2005), 243--288.

\bibitem[Bn]{Bn}  J.~Bowden, {\it Approximating $C^0$-foliations by contact structures}, Geo. Func. Anal. {\bf 26} (2016), 1255--1296.

\bibitem[BZ]{BZ} G.~Burde, H.~Zieschang, {\bf Knots}, Second edition. De Gruyter Studies in Mathematics, 5. Walter de Gruyter \& Co., Berlin, 2003.

\bibitem[Cal1]{Cal1} D.~Calegari, {\it Leafwise smoothing laminations}, Algebr. Geom. Topol. {\bf 1} (2001), 579--585.

\bibitem[Cal2]{Cal2} \bysame, {\bf Foliations and the geometry of 3-manifolds}, Oxford Mathematical Monographs, Oxford University Press, Oxford, 2007.

\bibitem[Can]{Can} A.~Candel, {\it Uniformization of surface laminations}, Ann. Sci. Ec. Norm. Sup. {\bf 26} (1993), 489--516.

\bibitem[CC]{CC} A.~Candel and L.~Conlon, {\bf Foliations II}, Graduate Studies in Mathematics 60, Amer. Math. Soc., 2003. 

\bibitem[CD]{CD} D.~Calegari and N.~Dunfield {\it Laminations and groups of homeomorphisms of the circle}, Inv. Math. {\bf 152} (2003), 149--204.

\bibitem[CLW]{CLW} A.~Clay, T.~Lidman and L.~Watson, {\it Graph manifolds, left-orderability and amalgamation}, Algebr. Geom. Topol. {\bf 13} (2013), 2347--2368.  

\bibitem[DPT]{DPT} M.K.~Dabkowski, J.H.~Przytycki, and A.A.~Togha, {\em Non-left-orderable 3-manifold groups}, Canad. Math. Bull. {\bf 48}(1) (2005), 32--40. 

\bibitem[Dun]{Dun} W.~Dunbar, {\em Geometric orbifolds}, Rev. Mat. Univ. Complut. Madrid {\bf 1}(1) (1988), 67--99.

\bibitem[FLP]{FLP} A.~Fathi, F.~Laudenbach, V.~ Po\'enaru, {\bf Thurston's work on surfaces}, Translated from the 1979 French original by Djun M. Kim and Dan Margalit. Mathematical Notes, 48. Princeton University Press, Princeton, NJ, 2012. 

\bibitem[FM]{FM} B.~Farb and D.~Margalit, {\bf A primer on mapping class groups}, Princeton University Press, 2011.

\bibitem[Ga1]{Ga1} D.~Gabai, {\it Foliations and the topology of 3-manifolds}, J. Differ. Geom. {\bf 18}(3) (1983), 445--503. 

\bibitem[Ga2]{Ga2} \bysame, {\it Foliations and $3$-manifolds}, in {\bf Proceedings of the {I}nternational {C}ongress of {M}athematicians}, {V}ol.\ {I}, {II} ({K}yoto, 1990), pages 609--619. Math. Soc. Japan, Tokyo, 1991.

\bibitem[GO]{GO} D.~Gabai and U.~Oertel, {\it Essential laminations in $3$-manifolds}, Ann. of Math. {\bf 130} (1989), 41--73.

\bibitem[Gei]{Gei} H. ~Geiges, {\bf An introduction to contact topology}, Volume {\bf 109},  Cambridge University Press, 2008.

\bibitem[Ghy]{Ghy} E.~Ghys, {\it Groups acting on the circle}, Ens. Math. {\bf 47} (2001), 329--408.

\bibitem[Ghi]{Ghi} P. ~Ghiggini, {\it Knot Floer homology detects genus-one fibred knots}, Amer. J. Math. {\bf 130}(5) (2008), 1151--1169. 

\bibitem[Gor]{Gor} C.~McA.~Gordon, {\it Riley's conjecture on $SL(2,\mathbb R)$ representations of $2$-bridge knots}, J. Knot Theory Ramifications {\bf 26} (2017), 6 pp. 

\bibitem[GLid1]{GLid1}
C.~McA.~Gordon and T.~Lidman, {\em Taut foliations, left-orderability, and cyclic branched covers}, Acta Math. Vietnam {\bf 39} (2014), no.4, 599--635.

\bibitem[GLid2]{GLid2}
\bysame, {\em Corrigendum to ``Taut foliations, left-orderability, and cyclic branched covers"}, Acta Math. Vietnam {\bf 42} (2017), 775--776. 

\bibitem[HRW]{HRW} J.~Hanselman, J.~Rasmussen and L.~Watson, {\it Bordered Floer homology for manifolds with torus boundary via immersed curves}, preprint (2016), arXiv:1604.03466.

\bibitem[HRRW]{HRRW} J.~Hanselman, J.~Rasmussen, S.~Rasmussen, and L.~Watson, {\it Taut foliations on graph manifolds}, preprint (2015), arXiv:1508.05911v1.

\bibitem[HKP]{HKP} S.~Harvey, K.~Kawamuro, and O.~Plamenevskaya, {\it On transverse knots and branched covers}, Int. Math. Res. Notices, {\bf 3} (2009), 512--546.

\bibitem[Hed]{Hed} M.~Hedden, {\it Notions of positivity and the {O}zsv\'ath-{S}zab\'o concordance invariant}, J. Knot Theory Ramifications {\bf 19} (2010), 617--629.

\bibitem[HKM1]{HKM1} K.~Honda, W.~Kazez, and G.~Mati{\'{c}}, {\it Rightveering diffeomorphisms of compact surfaces with boundary}, Invent. Math. {\bf 169} (2007), 427--449.

\bibitem[HKM2]{HKM2} \bysame, {\it Rightveering diffeomorphisms of compact surfaces with boundary II}, Geom. \& Top. {\bf 12} (2008), 2057--2094.

\bibitem[Hom]{Hom} J.~Hom, {\it Satellite knots and L-space surgeries}, Bull. London Math. Soc. {\bf 48} (2016), 771--778.

\bibitem[HS]{HS} J.~Howie and H.~Short, {\it The band-sum problem}, J. Lond. Math. Soc. {\bf 2} (1985), 571--576. 

\bibitem[Hu]{Hu} Y.~Hu, {\it Left-orderability and cyclic branched coverings}, Algebr. Geom. Topol.{\bf 15} (2015), 399--413.

\bibitem[Ito1]{Ito1} T.~Ito, {\it Braid ordering and knot genus}, J. Knot Theory Ramifications, {\bf 9} (2011),1311--1323.

\bibitem[Ito2]{Ito2} \bysame, {\it Braid ordering and the geometry of closed braid}, Geom. \& Top. {\bf 15} (2011), 473--498.

\bibitem[IK]{IK} T.~Ito and K.~Kawamuro, {\it Essential open book foliations and fractional Dehn twist coefficient}, Geom. Dedicata {\bf 187} (2017), 17--67.

\bibitem[Ju]{Ju} A.~Juh\'asz, {\it A survey of Heegaard Floer homology}, in {\bf New ideas in low dimensional topology}, 237--296. World Scientific, 2015.

\bibitem[KR1]{KR1} W. ~Kazez and R. ~Roberts, {\it Fractional Dehn twists in knot theory and contact topology}, Algebr. Geom. Topol.,{\bf 13} (2013), 3603--3637. 

\bibitem[KR2]{KR2} \bysame, {\it Approximating $C^{1,0}$-foliations}, in {\bf Interactions between low-dimensional topology and mapping class groups}, 21--72, Geom. Topol. Monogr., {\bf 19}, Geom. Topol. Publ., Coventry, 2015.

\bibitem[LS]{LS} P.~Lisca and A.~Stipsicz, {\it Ozsv\'ath-Szab\'o invariants and tight contact 3-manifolds, {III}}, J. Symplectic Geom. {\bf 5} (2007), 357--384.

\bibitem[LW]{LW} Y.~Li and L.~Watson, {\it Genus one open books with non-left-orderable fundamental group}, Proc. Amer. Math. Soc. {\bf 142} (2014), 1425--1435.

\bibitem[Mal]{Mal} A.~Malyutin, {\it Twist number of (closed) braids}, St. Petersburg Math. J. {\bf 16} (2005), 791--813.

\bibitem[MB]{MB} J.~W.~Morgan and H.~Bass, {\bf The Smith Conjecture}, Pure and Applied Math. {\bf 112}, Academic Press, 1984.

\bibitem[Mi]{Mil} J.~Milnor, {\it On the existence of a connection with curvature zero}, Comm. Math. Helv. {\bf 32} (1958), 215--223.

\bibitem[MM]{MM} S.~Matsumoto and S.~Morita, {\it Bounded cohomology of certain groups of homeomorphisms}, Proc. Amer. Math. Soc. {\bf 94} (1985), 539--544.

\bibitem[Mor]{Mor} S.~Morita, {\bf Geometry of Differential Forms}, Translations of Mathematical Monographs, {\bf 201}, American Mathematical Society, 2001.

\bibitem[MS]{MS} J.~Milnor, J.~Stasheff, {\bf Characteristic classes}, Annals of Mathematics Studies, No. 76. Princeton University Press, 1974.

\bibitem[MSY]{MSY} W.~Meeks, L.~Simon, and S.~T.~Yau, {\it Embedded minimal surfaces, exotic spheres, and manifolds with positive Ricci curvature}, Ann. Math. {\bf 116} (1982), 621--659.fibred

\bibitem[Mur]{Mur} K.~Murasugi, {\bf On closed 3-braids}, Mem.  Amer. Math. Soc. {\bf 151}, American Mathematical Society, Providence, R.I., 1974.

\bibitem[Ni]{Ni} Y.~Ni, {\it Knot Floer homology detects fibred knots}, Invent. Math. {\bf 170} (2007), 577--608.

\bibitem[Nov]{Nov} S.~P.~Novikov, {\it The topology of foliations}, Tru. Mosk. Mat. Obsc. {\bf 14} (1965), 248--278.

\bibitem[OS1]{OS1} P.~Ozsv\'ath and Z.~Szab\'o, {\it Holomorphic disks and genus bounds}, Geom. \& Top. {\bf 8} (2004), 311-- 334.

\bibitem[OS2]{OS2} \bysame, \emph{On knot Floer homology and lens space surgeries}, Topology {\bf 44} (2005), 1281--1300.

\bibitem[Pe]{Pe} T.~Peters, {\em On L-spaces and non-left-orderable 3-manifold groups}, preprint (2009), arXiv:0903.4495.

\bibitem[Pla]{Pla} J.~Plante, {\it Foliations with measure preserving holonomy}, Ann. Math. {\bf 102} (1975), 327--361.

\bibitem[Rob]{Rob} R.~Roberts, {\it Taut foliations in punctured surface bundles, II}, Proc. London Math. Soc. {\bf 83} (2001) 443--471. 

\bibitem[Rol]{Rol} D.~Rolfsen, {\bf Knots and links}, American Mathematical Soc. {\bf 346}, 2003.

\bibitem[Ros]{Ros} H.~Rosenberg, {\it Foliations by planes}, Topology {\bf 7} (1968), 131--138.

\bibitem[Spa]{Spa} E.~Spanier, {\bf Algebraic topology}, Springer-Verlag, 1995.

\bibitem[Thu1]{Thu1} W.~Thurston, {\it Hyperbolic structures on $3$-manifolds, II: Surface groups and $3$-manifolds which fiber over the circle}, preprint (1986), arXiv:math.GT/9801045.

\bibitem[Thu2]{Thu2} \bysame , {\it On the geometry and dynamics of diffeomorphisms of surfaces}, Bull. Amer. Math. Soc. {\bf 19} (1988), 417--431.

}

\end{thebibliography}
\end{document}